\documentclass[a4paper,11pt]{amsart}
\usepackage[pdftex]{color,graphicx}
\usepackage{amssymb}
\usepackage{amsfonts}
\usepackage{esint}
\usepackage{amsmath}
\usepackage{amsrefs}
\usepackage{epsfig}
\usepackage{amsthm}
\usepackage{color}
\usepackage[]{epsfig}
\usepackage[]{pstricks}
\usepackage[utf8]{inputenc}
\newtheorem{theorem}{Theorem}[section]
\newtheorem{proposition}{Proposition}[section]
\newtheorem{lemma}{Lemma}[section]

\newtheorem{example}{Example}[section]
\newtheorem{corollary}{Corollary}[section]

\newtheorem{remark}{Remark}[section]
\newcommand{\R}{\mathbb{R}}
\newcommand{\h}{\mathbb{H}}
\newcommand{\s}{\mathbb{S}}

\newcommand{\ria}{\rightarrow}

\newcommand{\n}{\nabla}
\newcommand{\ran}{\rangle}
\newcommand{\lan}{\langle}

\DeclareMathOperator{\hess}{Hess}
\DeclareMathOperator{\ric}{Ric}
\DeclareMathOperator{\scal}{scal}
\DeclareMathOperator{\di}{div}

\DeclareMathOperator{\diam}{diam}

\DeclareMathOperator{\lenght}{Lenght}
\DeclareMathOperator{\area}{Area}

\numberwithin{equation}{section}
\title[Isoperimetric inequalities for submanifolds in warped products manifolds]{Isoperimetric inequalities and monotonicity formulas for submanifolds in warped products manifolds}
\author{Hil\'ario Alencar \and Greg\'orio Silva Neto}
\date{April 11, 2017}

\address{Instituto de Matem\'atica\\
Universidade Federal de Alagoas\\
Macei\'o, AL, 57072-900, Brazil\\}
\email{hilario@mat.ufal.br}

\address{Instituto de Matem\'atica\\
Universidade Federal de Alagoas\\ 
Macei\'o, AL, 57072-900, Brazil\\}
\email{gregorio@im.ufal.br}

\begin{document}
\footnotetext{Hil\'ario Alencar was partially supported by CNPq of Brazil}

\subjclass[2010]{53C21, 53C42}

\begin{abstract}
In this paper we first prove some linear isoperimetric inequalities for submanifolds in the de Sitter-Schwarzschild and Reissner-Nordstrom manifolds. Moreover, the equality is attained. Next, we prove some monotonicity formulas for submanifolds with bounded mean curvature vector in warped product manifolds and, as consequences, we give lower bound estimates for the volume of these submanifolds in terms of the warping function. We conclude the paper with an isoperimetric inequality for minimal surfaces.
\end{abstract}

%
%

\maketitle

\section{Introduction}

Let $I\subset\R$ be an open interval and $N^{n-1}$ be a Riemannian manifold. We define the $n-$dimensional warped product manifold by $M^n=I\times N^{n-1}, \ n\geq 2,$ endowed with the warped metric
\begin{equation}\label{warped}
g = dr^2 + h(r)^2g_N,
\end{equation}
where $h:I\ria\R$ is a smooth and positive function, called the warping function, and $g_N$ is the metric of $N^{n-1}.$ 

These manifolds were first introduced by R. Bishop and B. O'Neill in 1969, see \cite{B-ON}, and has had increasing importance due its applications as model spaces in general relativity. There are many interesting papers in this subject, see for example \cite{Montiel}, \cite{Morgan}, \cite{Brendle-Eichmair}, \cite{Brendle}, \cite{Bessa}, \cite{Xia-Wu-1}, \cite{Xia-Wu-2}, \cite{Gimeno}, and \cite{Brendle-2} for more references and results.

In the following we introduce some examples of warped product manifolds used in this paper.

\begin{example}[The space forms $\R^n, \ \h^{n}(c), \ c<0,$ and $\s^{n}(c),\ c>0, n\geq 2$]\label{space-forms}
{\normalfont We can consider the space forms as warped product manifolds endowed with the warped metric $g=dr^2+h(r)^2g_{\s^{n-1}},$ where $g_{\s^{n-1}}$ denotes the standard metric of unit $(n-1)-$dimensional sphere $\s^{n-1}.$
\begin{itemize}
\item[(i)] For $\R^n,$ the warping function is $h(r)=r,$ $r\in(0,\infty);$
\item[(ii)] For $\h^n(c),$ the warping function is $h(r)=\frac{1}{\sqrt{-c}}\sinh(\sqrt{-c}r),$ $r\in(0,\infty);$
\item[(iii)] For $\s^n(c),$ the warping function is $h(r)=\frac{1}{\sqrt{c}}\sin(\sqrt{c}r),$ $r\in(0,\pi).$
\end{itemize}

}
\end{example}

\begin{example}[The de Sitter-Schwarzschild manifolds]{\normalfont Let $n\geq 3,$ $m>0$ and $c\in\R.$ Let \[(s_0,s_1)=\{s>0 ; 1-ms^{2-n}-cs^2>0\}.\] If $c\leq 0,$ then $s_1=\infty.$ If $c>0,$ assume that $\frac{n^n}{4(n-2)^{n-2}}m^2c^{n-2}<1.$ The de Sitter-Schwarzschild manifold is defined by $M^n(c)=(s_0,s_1)\times \s^{n-1}$ endowed with the metric
\[
g=\dfrac{1}{1-ms^{2-n}-cs^2}ds^2 + s^2 g_{\s^{n-1}}.
\]
In order to write the metric $g$ in the form (\ref{warped}), define $F:[s_0,s_1)\ria \R$ by
\[
F'(s)=\dfrac{1}{\sqrt{1-ms^{2-n}-cs^2}}, \ F(s_0)=0.
\]
Taking $r=F(s),$ we can write $g=dr^2+h(r)^2g_{\s^{n-1}},$ where $h:[0,F(s_1))\ria[s_0,s_1)$ denotes the inverse function of $F.$ The function $h(r)$ clearly satisfies
\begin{equation}\label{defi-SS}
h'(r)=\sqrt{1-mh(r)^{2-n}-ch(r)^2},\ h(0)=s_0,\ \mbox{and}\ h'(0)=0.
\end{equation}
}
\end{example}

\begin{example}[The Reissner-Nordstrom manifold]\label{ex-RN}{\normalfont The Reissner-Nordstrom manifold is defined by $M^n=(s_0,\infty)\times\s^{n-1},\ n\geq 3,$ with the metric
\[
g=\dfrac{1}{1-ms^{2-n}+q^2s^{4-2n}}ds^2 + s^2 g_{\s^{n-1}},
\]
where $m>2q>0$ and $s_0=\left(\frac{2q^2}{m-\sqrt{m^2-4q^2}}\right)^{\frac{1}{n-2}}$ is the larger of the two solutions of $1-ms^{2-n}+q^2s^{4-2n}=0.$ In order to write the metric $g$ in the form (\ref{warped}), define $F:[s_0,\infty)\ria \R$ by
\[
F'(s)=\dfrac{1}{\sqrt{1-ms^{2-n}+q^2s^{4-2n}}}, \ F(s_0)=0.
\]
Taking $r=F(s),$ we can write $g=dr^2+h(r)^2g_{\s^{n-1}},$ where $h:[0,\infty)\ria[s_0,\infty)$ denotes the inverse function of $F.$ The function $h(r)$ clearly satisfies
\begin{equation}\label{defi-RN}
h'(r)=\sqrt{1-mh(r)^{2-n}+q^2h(r)^{4-2n}},\ h(0)=s_0,\ \mbox{and}\ h'(0)=0.
\end{equation}
}
\end{example}

In this paper we first prove some linear isoperimetric inequalities for submanifolds in the de Sitter-Schwarzschild and Reissner-Nordstrom manifolds. In particular, we obtain some known isoperimetric inequalities for submanifolds in space forms. Next, we prove some monotonicity formulas for submanifolds with bounded mean curvature vector in warped product manifolds and, as consequences, we give lower bound estimates for the volume of these submanifolds in terms of the warping function. We conclude the paper with an isoperimetric inequality for minimal surfaces.  

The first result is an isoperimetric inequality for submanifolds in the de Sitter-Schwarzschild manifold.

\begin{theorem}\label{iso-SS1} Let $\Sigma$ be a $k-$dimensional, compact, oriented, submanifold, $k\geq 2$, of the de Sitter-Schwarzschild manifold $M^n(c)=(s_0, s_1)\times \s^{n-1}, \ n\geq3.$
\begin{itemize}
\item[(i)] If $\Sigma\subset \left(s_0,\left(\frac{mn}{2}\right)^{\frac{1}{n-2}}\right)\times\s^{n-1},$ then 
\begin{equation}\label{iso-SS1-d}
\begin{split}
|\Sigma|&\leq \dfrac{d_\Sigma}{k\sqrt{1-md_\Sigma^{2-n}-cd_\Sigma^2}}\left[|\partial\Sigma| + k\int_\Sigma |\vec{H}| d\Sigma\right]\\
&\qquad -\frac{d_\Sigma^2}{k(n-1)(1-md_\Sigma^{2-n}-cd_\Sigma^2)}\int_\Sigma\ric_{M^n(c)}(\n r)|\n_\Sigma r|^2d\Sigma,
\end{split}
\end{equation}
where $d_\Sigma=\min\left\{s\in (s_0,s_1);\Sigma\cap \{\{s\}\times \s^{n-1}\}\neq\emptyset\right\}.$ In particular,
\begin{equation}\label{iso-SS1-mod-3}
|\Sigma|\leq C_1(d_\Sigma)\left[|\partial\Sigma| + k\int_\Sigma |\vec{H}| d\Sigma\right],
\end{equation}
where
\[
C_1(d_\Sigma) = \dfrac{d_\Sigma\sqrt{1-md_\Sigma^{2-n}-cd_\Sigma^2}}{(1-\frac{mn}{2}d_\Sigma^{2-n})+(k-1)(1-md_\Sigma^{2-n}-cd_\Sigma^2)}.
\]

\item[(ii)] If $c>0$ and $\Sigma\subset\left(\left(\frac{m(n-2)}{2c}\right)^{\frac{1}{n}},s_1\right)\times\s^{n-1},$ then
\begin{equation}\label{iso-SS1-R0}
\begin{split}
|\Sigma|&\leq \dfrac{R_\Sigma}{k\sqrt{1-mR_\Sigma^{2-n}-cR_\Sigma^2}}\left[|\partial\Sigma| + k\int_\Sigma |\vec{H}| d\Sigma\right]\\
&\qquad -\frac{d_\Sigma^2}{k(n-1)(1-md_\Sigma^{2-n}-cd_\Sigma^2)}\int_\Sigma\ric_{M^n(c)}(\n r)|\n_\Sigma r|^2d\Sigma,
\end{split}
\end{equation}
where $R_\Sigma=\max\left\{s\in  \left(s_0,s_1\right);\Sigma\cap \{\{s\}\times \s^{n-1}\}\neq\emptyset\right\}.$

\item[(iii)] If $c\leq 0$ and $\Sigma\in \left(\left(\frac{mn}{2}\right)^{\frac{1}{n-2}},\infty\right)\times\s^{n-1},$ or $c>0$ and \\$\Sigma\subset\left(\left(\frac{mn}{2}\right)^{\frac{1}{n-2}},\left(\frac{m(n-2)}{2c}\right)^{\frac{1}{n}}\right)\times\s^{n-1},$ then
\begin{equation}\label{iso-SS1-R}
\begin{split}
|\Sigma|&\leq \dfrac{R_\Sigma}{k\sqrt{1-mR_\Sigma^{2-n}-cR_\Sigma^2}}\left[|\partial\Sigma| + k\int_\Sigma |\vec{H}| d\Sigma\right]\\
&\qquad -\frac{R_\Sigma^2}{k(n-1)(1-mR_\Sigma^{2-n}-cR_\Sigma^2)}\int_\Sigma\ric_{M^n(c)}(\n r)|\n_\Sigma r|^2d\Sigma.
\end{split}
\end{equation}
\item[(iv)] For $c\in \R$ and $\Sigma\subset\left(\left(\frac{mn}{2}\right)^{\frac{1}{n-2}},s_1\right)\times\s^{n-1},$ we have also
\begin{equation}\label{iso-SS1-mod}
|\Sigma|\leq \dfrac{R_\Sigma}{(k-1)\sqrt{1-mR_\Sigma^{2-n}-cR_\Sigma^2}}\left[|\partial\Sigma| + k\int_\Sigma |\vec{H}| d\Sigma\right].
\end{equation}
In particular, if $c<0,$ then
\begin{equation}\label{iso-SS1-mod-2}
|\Sigma|\leq \dfrac{1}{\sqrt{-c}(k-1)}\left[|\partial\Sigma| + k\int_\Sigma |\vec{H}| d\Sigma\right].
\end{equation}
\end{itemize}
Moreover, if $\Sigma$ is a slice $\{s\}\times\s^{n-1},$ then the equality holds for the inequalities (\ref{iso-SS1-d}), (\ref{iso-SS1-R0}), and (\ref{iso-SS1-R}). Here, $\ric_{M^n(c)}$ denotes the Ricci curvature of $M^n(c),$ $\n r$ denotes the gradient of the distance function $r,$ $\n_\Sigma r$ denotes the component of $\n r$ tangent to $\Sigma,$ and $\vec{H}$ denotes the normalized mean curvature vector of $\Sigma.$ 
\end{theorem}

\begin{remark}
{\normalfont
Since $C_1(d_\Sigma)<0$ for $d_\Sigma$ near from $s_0,$ the inequality (\ref{iso-SS1-mod-3}) holds only away from $s_0,$ for those $d_\Sigma$ such that $C_1(d_\Sigma)>0.$ Since $C_1\left(\left(\frac{mn}{2}\right)^{\frac{1}{n-2}}\right)=\frac{1}{k-1}>0,$ there exists $\overline{s}\in\left(s_0,\left(\frac{mn}{2}\right)^{\frac{1}{n-2}}\right),$ depending on $m,n,c$ and $k,$ such that $C_1(d_\Sigma)>0$ for every $d_\Sigma\in \left(\overline{s},\left(\frac{mn}{2}\right)^{\frac{1}{n-2}}\right).$
}
\end{remark}

If $k=n,$ we obtain an isoperimetric inequality for domains in the de Sitter-Schwarzschild manifold:

\begin{corollary}
Let $\Omega$ be a domain of the de Sitter-Schwarzschild manifold with smooth boundary $\partial\Omega.$
\begin{itemize}
\item[(i)] If $\Omega \subset \left(s_0,\left(\frac{mn}{2}\right)^{\frac{1}{n-2}}\right)\times \s^{n-1},\ n\geq3,$ then 
\begin{equation}\label{iso-SS1-domain-1}
|\Omega|\leq C_1(d_\Omega)|\partial\Omega|,
\end{equation}
where 
\[
C_1(d_\Omega) = \dfrac{d_\Omega\sqrt{1-md_\Omega^{2-n}-cd_\Omega^2}}{(1-\frac{mn}{2}d_\Omega^{2-n})+(n-1)(1-md_\Omega^{2-n}-cd_\Omega^2)},
\]
and
$d_\Omega=\min\left\{s\in\left(s_0,\left(\frac{mn}{2}\right)^{\frac{1}{n-2}}\right);\Omega\cap \{\{s\}\times \s^{n-1}\}\neq\emptyset\right\}.$
\item[(ii)] If $\Omega\subset\left(\left(\frac{mn}{2}\right)^{\frac{1}{n-2}},s_1\right)\times\s^{n-1},\ n\geq3,$ then
\begin{equation}\label{iso-SS1-domain-2}
|\Omega|\leq \dfrac{R_\Omega}{(n-1)\sqrt{1-mR_\Omega^{2-n}-cR_\Omega^2}}|\partial\Omega|,
\end{equation}
where $R_\Omega=\max\left\{s\in  \left(\left(\frac{mn}{2}\right)^{\frac{1}{n-2}},s_1\right);\Omega\cap \{\{s\}\times \s^{n-1}\}\neq\emptyset\right\}.$ In particular, if $c<0,$ then
\begin{equation}\label{iso-SS1-domain-3}
|\Omega|\leq\frac{1}{\sqrt{-c}(n-1)}|\partial\Omega|.
\end{equation}

\end{itemize} 
\end{corollary}

Taking $m\ria 0$ in the de Sitter-Schwarzschild manifold, it becomes $\h^n(c)$ for $c<0,$ $\s^n(c)$ for $c>0$ and $\R^n$ for $c=0.$ Thus, as consequences of Theorem \ref{iso-SS1} we obtain isoperimetric inequalities for submanifolds in space forms. First, we present an isoperimetric inequality for submanifolds of the hyperbolic space.

\begin{corollary}\label{iso-Hn}
Let $\Sigma$ be a $k-$dimensional, compact, oriented, submanifold, $k\geq2,$ of the hyperbolic space $\h^n(c),\ n\geq3, \ c<0.$ Then
\begin{equation}\label{ineq-iso-Hn}
|\Sigma| \leq \frac{\tanh(\sqrt{-c}\widetilde{R}_\Sigma)}{\sqrt{-c}k}\left[|\partial\Sigma| + k\int_\Sigma |\vec{H}| d\Sigma\right] + \frac{1}{k}\int_\Sigma\tanh^2(\sqrt{-c}r)|\n_\Sigma r|^2d\Sigma,
\end{equation}
where $\widetilde{R}_\Sigma$ is the radius of the smallest extrinsic ball which contains $\Sigma.$ If $\Sigma$ is a geodesic sphere, then the equality holds.
\end{corollary}

\begin{remark}
{\normalfont
It is possible to obtain another proof of Corollary \ref{iso-Hn} from the proofs of Theorem 6 (b), p. 185 of \cite{CG-Manusc}, for $\vec{H}=0$ and Corollary 3.6, p.533 of \cite{Seo}, for arbitrary $\vec{H}.$
}
\end{remark}

\begin{remark}
{\normalfont Since the de Sitter-Schwarzschild manifold, $c<0,$ becomes $\h^n(c)$ when $m\ria 0,$ the inequality (\ref{iso-SS1-domain-3}) holds also for $\h^n(c).$ This fact was proved first by S.-T.Yau in \cite{Yau}, see Proposition 3, p.498.
}
\end{remark}

The next corollary is an isoperimetric inequality for submanifolds of the open hemisphere $\s^n_+(c).$ 

\begin{corollary}\label{iso-Sn}
Let $\Sigma$ be a $k-$dimensional, compact, oriented, submanifold, $k\geq2,$ of the open hemisphere $\s^n_+(c),\ n\geq3, \ c>0.$ Then
\begin{equation}\label{ineq-iso-Sn}
|\Sigma| \leq \frac{\tan(\sqrt{c}\widetilde{R}_\Sigma)}{k\sqrt{c}}\left[|\partial\Sigma| + k\int_\Sigma |\vec{H}| d\Sigma\right] - \frac{1}{k}\int_\Sigma\tan^2(\sqrt{c}r)|\n_\Sigma r|^2 d\Sigma,
\end{equation}
where $\widetilde{R}_\Sigma$ is the radius of the smallest extrinsic ball which contains $\Sigma.$ If $\Sigma$ is a geodesic sphere, then the equality holds.
\end{corollary}

\begin{remark}
{\normalfont
It is possible to obtain another proof of Corollary \ref{iso-Hn} from the proofs of Theorem 6 (a), p. 185, of \cite{CG-Manusc}, for $\vec{H}=0$ and Corollary 3.4, p.533, of \cite{Seo}, for arbitrary $\vec{H}.$
}
\end{remark}

Our second result is the following isoperimetric inequality for submanifolds in the Reissner-Nordstrom manifold:

\begin{theorem}\label{iso-RN} Let $\Sigma$ be a $k-$dimensional, compact, oriented, submanifold, $k\geq2,$ of the Reissner-Nordstrom manifold $M^n=(s_0,\infty)\times\s^{n-1},$ $\ n\geq3.$ Let $s_2\!=\!\left(\frac{4(n-1)q^2}{mn -\sqrt{m^2n^2 - 16(n-1)q^2}}\right)^{\frac{1}{n-2}}.$
\begin{itemize}
\item[(i)] If $\Sigma\subset(s_0,s_2)\times\s^{n-1},$ then
\begin{equation}\label{iso-RN-d1}
\begin{split}
|\Sigma|&\leq \frac{d_\Sigma}{k\sqrt{1-md_\Sigma^{2-n}+q^2d_\Sigma^{4-2n}}}\left[|\partial\Sigma| + k\int_\Sigma |\vec{H}|d\Sigma\right] \\
&\qquad- \frac{d_{\Sigma}^2}{k(n-1)(1-md_{\Sigma}^{2-n}+q^2d_{\Sigma}^{4-2n})}\int_\Sigma\ric_{M^n}(\n r)|\n_\Sigma r|^2 d\Sigma.
\end{split}
\end{equation}
In particular,
\begin{equation}\label{iso-RN-mod-1}
|\Sigma|\leq \frac{d_\Sigma}{(C_2(d_\Sigma) - k)\sqrt{1-md_\Sigma^{2-n}+q^2d_\Sigma^{4-2n}}}\left[|\partial\Sigma| + k\int_\Sigma |\vec{H}|d\Sigma\right],
\end{equation}
where $d_\Sigma=\min\left\{s\in(s_0,\infty);\Sigma\cap \{\{s\}\times \s^{n-1}\}\neq\emptyset\right\}.$
\item[(ii)] If $\Sigma\subset\left(s_2, \infty \right)\times\s^{n-1},$ then 
\begin{equation}\label{iso-RN-R}
\begin{split}
|\Sigma|&\leq \frac{R_\Sigma}{k\sqrt{1-mR_\Sigma^{2-n}+q^2R_\Sigma^{4-2n}}}\left[|\partial\Sigma| + k\int_\Sigma |\vec{H}|d\Sigma\right]\\
&\qquad - \frac{R_{\Sigma}^2}{k(n-1)(1-mR_{\Sigma}^{2-n}+q^2R_{\Sigma}^{4-2n})}\int_\Sigma\ric_{M^n}(\n r)|\n_\Sigma r|^2 d\Sigma.
\end{split}
\end{equation}
In particular,
\begin{equation}\label{iso-RN-mod-3}
|\Sigma|\leq \frac{R_\Sigma}{(C_2(d_\Sigma) - k)\sqrt{1-mR_\Sigma^{2-n}+q^2R_\Sigma^{4-2n}}}\left[|\partial\Sigma| + k\int_\Sigma |\vec{H}|d\Sigma\right],
\end{equation}
where $R_\Sigma=\max\left\{s\in  (s_2,\infty);\Sigma\cap \{\{s\}\times \s^{n-1}\}\neq\emptyset\right\}.$
\end{itemize}
Moreover, if $\Sigma$ is a slice $\{s\}\times\s^{n-1}$ then the equality holds for both inequalities (\ref{iso-RN-d1}) and (\ref{iso-RN-R}).
Here,
\[
C_2(d_\Sigma)=\frac{n-2}{2 d_\Sigma^{n-2}}\!\!\left(m-\frac{2q^2}{d_\Sigma^{n-2}}\right)\frac{1}{1-md_\Sigma^{2-n}+q^2d_\Sigma^{4-2n}} = \frac{m(n-2)}{2}d_\Sigma^{2-n} + O(d_\Sigma^{4-2n}).
\]
\end{theorem}

\begin{remark}
{\normalfont
Since $C_2(d_\Sigma)\ria\infty$ when $d_\Sigma\ria s_0,$ the inequalities (\ref{iso-RN-mod-1}) and (\ref{iso-RN-mod-3}) holds only away from $s_0,$ i.e., for $C_2(d_\Sigma)<k.$ On the other hand, since $C_2(d_\Sigma)\ria 0$ when $d_\Sigma\ria\infty,$ there exists $\overline{s}\in(s_0,\infty)$ depending on $m,q,n$ and $k,$ such that $C_2(d_\Sigma)<k$ for every $d_\Sigma\in (\overline{s},\infty).$
}
\end{remark}

If $k=n,$ we obtain an isoperimetric inequality for domains in the Reissner-Nordstrom manifold.

\begin{corollary}
Let $\Omega$ be a compact domain in the Reissner-Nordstrom manifold with smooth boundary $\partial\Omega.$ 

\begin{itemize}
\item[(i)] If $\Omega\subset(s_0,s_2)\times\s^{n-1},\ n\geq3,$ then
\[
|\Omega|\leq \frac{d_\Omega}{(C_2(d_\Omega) - n)\sqrt{1-md_\Omega^{2-n}+q^2d_\Omega^{4-2n}}}|\partial\Omega|,
\]
where $d_\Omega=\min\left\{s\in(s_0,\infty);\Omega\cap \{\{s\}\times \s^{n-1}\}\neq\emptyset\right\}.$
\item[(ii)] If $\Omega\subset\left(s_2, \infty \right)\times\s^{n-1},\ n\geq3,$ then 
\[
|\Omega|\leq \frac{R_\Omega}{(C_2(d_\Omega) - n)\sqrt{1-mR_\Omega^{2-n}+q^2R_\Omega^{4-2n}}}|\partial\Omega|,
\]
where $R_\Omega=\max\left\{s\in  (s_2,\infty);\Omega\cap \{\{s\}\times \s^{n-1}\}\neq\emptyset\right\}.$
\end{itemize}
\end{corollary}

\begin{remark}
{\normalfont 
If $\Sigma$ is a compact, without boundary, embedded, orientable hypersurface of the de Sitter-Schwarzschild manifold or the Reissner-Nordstrom manifold, with constant mean curvature, S. Brendle in \cite{Brendle}, see Corollary 1.2 and Corollary 1.3, pp. 249-250, proved that $\Sigma$ is a slice.
}
\end{remark}
If we suppose in addition that the norm of the mean curvature vector is bounded, we obtain the following monotonicity formula for submanifolds in warped product manifolds. Let $M^n=I\times N^{n-1}$ be the warped product manifold. Hereafter, we denote by \[B_r = \{s\in I| s\leq r\}\times N^{n-1}\subset M^n.\]

\begin{theorem}\label{mono}
Let $I\subset\R$ be an open interval and $N^{n-1},\ n\geq3,$ be a $(n-1)-$dimensional Riemannian manifold. Let $M^n=I\times N^{n-1}$ be endowed with the warped metric $ds^2=dr^2+h(r)^2g_N,$ where $g_N$ is the metric of $N^{n-1}$ and $h'(r)>0$ for all $r\in I.$ Assume also that $\dfrac{h(r)}{h'(r)}$ is non-decreasing for all $r\in I.$ If $\Sigma$ is a $k-$dimensional, proper, oriented, submanifold of $M^n$ such that its mean curvature vector satisfies $k|\vec{H}|\leq \alpha$ for some $\alpha\geq0,$ then 

\begin{itemize}

\item[(i)] the function $V_1:I\ria\R$ given by
\[
V_1(r)=\dfrac{e^{\alpha r}}{h(r)^k}\int_{\Sigma\cap B_r}\!\! h(s) d\Sigma
\]
is monotone non-decreasing. Moreover, 
\begin{equation}\label{est.mono-1}
|\Sigma\cap B_r|\geq C_1(r_0)e^{-\alpha r}h(r)^{k-1},
\end{equation}
for every $r>r_0,\ r_0,r\in I,$ where $\displaystyle{C_1(r_0)=\frac{e^{\alpha r_0}}{h(r_0)^k}\int_{\Sigma\cap B_{r_0}}h(s)d\Sigma};$
\item[(ii)] the function $V_2:I\ria\R$ given by
\[
V_2(r)=\dfrac{e^{\alpha r}}{h(r)^k}\int_{\Sigma\cap B_r}\!\! h'(s) d\Sigma
\]
is monotone non-decreasing. Moreover, if $h'(r)\leq B,\ B>0,$ then 
\begin{equation}\label{est.mono-2}
|\Sigma\cap B_r|\geq \frac{C_2(r_0)}{B}e^{-\alpha r}h(r)^{k},
\end{equation}
and if $h''(r)>0$ then
\begin{equation}\label{est.mono-3}
|\Sigma\cap B_r|\geq \frac{C_2(r_0)}{h'(r)}e^{-\alpha r}h(r)^{k},
\end{equation}
for every $r>r_0,\ r_0,r\in I,$ where $\displaystyle{C_2(r_0)=\frac{e^{\alpha r_0}}{h(r_0)^k}\int_{\Sigma\cap B_{r_0}}h'(s)d\Sigma}.$
\end{itemize}
\end{theorem} 
In the next corollary we assume 
\begin{equation}\label{star}
\lan\vec{H},\n r\ran\geq 0.
\end{equation} 
This condition holds, for example, for minimal submanifolds ($\vec{H}=0$) or cones ($\lan\vec{H},\n r\ran=0$) in the warped product manifolds. If $k=n-1,$ we say that $\Sigma$ is a star-shaped hypersurface if there is a choice of unit normal $\eta$ of $\Sigma$ such that $\lan\eta,\n r\ran\geq0$. In this case, the condition (\ref{star}) means that $\Sigma$ is star-shaped.
\begin{corollary}\label{mono-cones}
Let $I\subset\R$ be an open interval and $N^{n-1},\ n\geq3,$ be a $(n-1)-$dimensional Riemannian manifold. Let $M^n=I\times N^{n-1}$ be endowed with the warped metric $ds^2=dr^2+h(r)^2g_N,$ where $g_N$ is the metric of $N^{n-1}$ and $h'(r)>0$ for all $r\in I.$ Assume also that $\dfrac{h(r)}{h'(r)}$ is non-decreasing for all $r\in I.$ If $\Sigma$ is a $k-$dimensional, proper, oriented, submanifold, $k\geq2,$ of $M^n$ such that $\lan\vec{H},\n r\ran\geq 0,$ then the functions 
\[
r\longmapsto \dfrac{1}{h(r)^k}\int_{\Sigma\cap B_r}\!\! h(s) d\Sigma\ \ \mbox{and}\ \ r\longmapsto \dfrac{1}{h(r)^k}\int_{\Sigma\cap B_r}\!\! h'(s) d\Sigma
\]
are monotone non-decreasing for all $r\in I.$ In particular,
\begin{equation}\label{est.mono-cones-1}
|\Sigma\cap B_r|\geq \widetilde{C}_1(r_0)h(r)^{k-1}.
\end{equation}
Moreover, if there exists $B>0$ such that $h'(r)\leq B$ for every $r\in I,$ then
\begin{equation}\label{est.mono-cones-2}
|\Sigma\cap B_r|\geq \frac{\widetilde{C}_2(r_0)}{B}h(r)^{k}
\end{equation}
and, if $h''(r)>0$ then
\[
|\Sigma\cap B_r|\geq \frac{\widetilde{C}_2(r_0)}{h'(r)}h(r)^{k}
\]
for every $r>r_0,$ where 
\[
\widetilde{C}_1(r_0)=\frac{1}{h(r_0)^k}\int_{\Sigma\cap B_{r_0}}h(s)d\Sigma\ \mbox{and}\ \widetilde{C}_2(r_0)=\frac{1}{h(r_0)^k}\int_{\Sigma\cap B_{r_0}}h'(s)d\Sigma.
\]
\end{corollary}

As applications of Theorem \ref{mono} we have the following results:
\begin{corollary}\label{mono-SS}
Let $\Sigma$ be a $k-$dimensional, proper, oriented, submanifold, $k\geq 2,$ of the de Sitter-Schwarzschild manifold $M^n(c),\ n\geq 3,$ such that $k|\vec{H}|\leq \alpha$ for some $\alpha\geq0.$ Then, for every $r>r_0$ such that $h(r_0)>\left(\frac{mn}{2}\right)^{\frac{1}{n-2}},$ we have

\begin{itemize}
\item[(i)] for $c>0,$
\[
|\Sigma\cap B_r| \geq C_2(r_0)e^{-\alpha r}h(r)^k,
\]
where $\displaystyle{C_2(r_0)=\frac{e^{\alpha r_0}}{h(r_0)^k}\int_{\Sigma\cap B_{r_0}} h'(s) d\Sigma};$

\item[(ii)] for $c\leq 0,$
\[
|\Sigma\cap B_r| \geq C_1(r_0)e^{-\alpha r}h(r)^{k-1};
\]
where $\displaystyle{C_1(r_0)=\frac{e^{\alpha r_0}}{h(r_0)^k}\int_{\Sigma\cap B_{r_0}} h(s) d\Sigma}.$ Moreover, for $c<0,$
\begin{equation}\label{brendle-SS}
h(r)= \frac{1}{\sqrt{-c}}\sinh(\sqrt{-c}r)+\dfrac{m}{2n\sqrt{-c}}\sinh^{1-n}(\sqrt{-c}r) + O(\sinh^{-n-1}(\sqrt{-c}r)).
\end{equation} 
In particular, if $c<0,$ $\Sigma$ is complete, non compact, and $\alpha< k-1,$ then $\Sigma$ has at least exponential volume growth at infinity and $|\Sigma|=\infty.$
\end{itemize}
\end{corollary}
\begin{remark}
{\normalfont
The equation (\ref{brendle-SS}), with $c=-1,$ was proved by S. Brendle, see \cite{Brendle-2}, Lemma 2.1, p.128.
}
\end{remark}
For submanifolds of the Reissner-Nordstrom manifold, we have the
\begin{corollary}\label{mono-RN}
Let $\Sigma$ be a $k-$dimensional, proper, oriented, submanifold, $k\geq2,$ of the Reissner-Nordstrom manifold $M^n=(s_0,\infty)\times\s^{n-1}, \ n\geq3,$ such that $k|\vec{H}|\leq \alpha$ for some $\alpha\geq0.$ Then, for every $r>r_0$ such that $h(r_0)>s_2,$ we have
\[
|\Sigma\cap B_r| \geq C_2(r_0)e^{-\alpha r}h(r)^k,
\]
where $s_2=\left(\frac{4q^2(n-1)}{mn - \sqrt{m^2n^2 - 16q^2(n-1)}}\right)^{\frac{1}{n-2}},$ $C_2(r_0)=\displaystyle{\frac{e^{\alpha r_0}}{h(r_0)^k}\int_{\Sigma\cap B_{r_0}} h'(s) d\Sigma}.$ Moreover, if $n\geq4,$ then
\[
h(r) = r + \frac{m}{2(n-3)}r^{3-n}+ O(r^{5-2n}).
\]
In particular, if $n\geq4$ and $\Sigma$ is a complete minimal submanifold, then the volume of $\Sigma$ has at least polynomial growth of order $k$ at infinity and $|\Sigma|=\infty.$

\end{corollary}

Another interesting application of Theorem \ref{mono} is for warped manifolds $I\times\s^{n-1}$ where $I=(0,b)$ or $I=(0,\infty)$ which warping function satisfies $h(0)=0$ and $h'(0)=1.$ 

\begin{corollary}\label{mono-space}
Let $I\subset\R$ be an open interval of the form $(0,b)$ or $(0,\infty),$ $N^{n-1},\ n\geq3,$ be a Riemannian manifold and let $M^n=I\times N^{n-1}$ endowed with the metric $ds^2=dr^2 + h(r)^2g_N,$ such that $h(0)=0, \ h'(0)=1$ and $h'(r)>0$ for every $r>0.$ Assume also that $\dfrac{h(r)}{h'(r)}$ is non-decreasing for all $r\in I.$ If $\Sigma$ is a $k-$dimensional, proper, oriented, submanifold of $M^n$ such that its mean curvature vector satisfies $k|\vec{H}|\leq \alpha$ for some $\alpha\geq0,$ then 
\[
\int_{\Sigma\cap B_r(x_0)} h'(s)d\Sigma \geq \omega_ke^{-\alpha r}h(r)^k,
\]
for all $B_r(x_0)\subset M^n$ such that $x_0\in\Sigma,$ where $\omega_k$ is the volume of the $k$-dimensional Euclidean unit round ball. In particular, if there exists $B>0$ such that $h'(r)\leq B$ for every $r\in I,$ then 
\begin{equation}\label{mono-space-1}
|\Sigma\cap B_r| \geq \frac{\omega_k}{B}e^{-\alpha r}h(r)^k,
\end{equation}
and if $h''(r)>0$ then
\begin{equation}\label{mono-space-2}
|\Sigma\cap B_r| \geq \omega_k\frac{e^{-\alpha r}h(r)^k}{h'(r)}.
\end{equation}
\end{corollary}

\begin{remark}
{\normalfont
Clearly the space forms $\R^n, \ \h^n(c),$ and the open hemisphere $\s_+^n(c)$ satisfy the hypothesis of Theorem \ref{mono} and Corollary \ref{mono-space}. Other classes of manifolds satisfying these hypothesis are given in the Example \ref{cylindrical} and the Example \ref{ex-log} in the Appendix.
}
\end{remark}

When the submanifolds have dimension $2,$ we obtain other type of isoperimetric inequality, namely:

\begin{theorem}\label{Theo.2}
Let $I\subset\R$ be an open interval of the form $(0,b)$ or $(0,\infty),$ $N^{n-1},\ n\geq3,$ be a Riemannian manifold and let $M^n=I\times N^{n-1}$ endowed with the metric $ds^2=dr^2 + h(r)^2g_N,$ such that $h(0)=0, \ h'(0)=1$ and $h'(r)>0$ for every $r>0.$ Let $\Sigma^2\subset[r_0,r_1]\times N^{n-1}\subset M^n, \ [r_0,r_1]\subset I,$ be a compact minimal surface with non-empty boundary $\partial\Sigma.$
\begin{itemize}
\item[(i)] If the function $u(r):=r+\dfrac{h(r)}{h'(r)}$ is non-decreasing for $r\in[r_0,r_1],$ then
\[
2\pi A \leq L^2 + \dfrac{A}{n-1}\int_\Sigma \ric_M(\n r)d\Sigma;
\]
\item[(ii)] If the function $u(r):=r+\dfrac{h(r)}{h'(r)}$ is non-increasing for $r\in[r_0,r_1]$ and the scalar curvature of $N$ satisfies $\scal_N\geq0,$ then
\[
2\pi A \leq L^2 + \dfrac{2A}{(n-1)(n-2)}\int_\Sigma (\scal_M - 2\ric_M(\n r))d\Sigma,
\]
\end{itemize} 
where $A = \area(\Sigma),\ L=\lenght(\partial\Sigma),$ $\scal_M$ denotes the scalar curvature of $M,$ and $\ric_M(\n r)$ denotes the Ricci curvature of $M$ in the radial direction $\n r.$
\end{theorem}

\begin{remark}
{\normalfont
The function $u(r)=r+\dfrac{h(r)}{h'(r)}$ in the hypothesis of Theorem \ref{Theo.2} deserves some comments about where it is non-decreasing or non-increasing:

\begin{itemize}
\item[(a)] Since $u'(0)=2,$ there exists an interval $[0,s]\subset I$ such that $u'(r)>0;$

\item[(b)] There are manifolds where $u'(r)>0$ everywhere, as we can see in the space forms $\R^n,\ \s^n(c),$ $\h^n(c),$ the Example \ref{cylindrical}, and the Example \ref{ex-log};

\item[(c)] If $u'(r)<0$ somewhere, this happens only for compact intervals. In fact, if $u'(r)<0$ for $r>r_0,$ then $0<u(r)<u(r_0)$ implies $r<r_0 +\frac{h(r_0)}{h'(r_0)}.$ A typical case is the Example \ref{ex-theo.2-1} of the Appendix, where there exist $r_0$ and $r_1$ such that $u'(r)>0$ for $r\in(0,r_0)\cup (r_1,\infty)$ and $u'(r)<0$ for $r\in(r_0,r_1).$

\end{itemize}
}
\end{remark}

\begin{remark}
{\normalfont The hypothesis of Theorem \ref{Theo.2} are not satisfied by the de Sitter-Schwarzschild manifold nor by the Reissner-Nordstrom manifold.
}
\end{remark}

As immediate consequences of the item (i) of Theorem \ref{Theo.2}, we obtain the isoperimetric inequalities of J. Choe and R. Gulliver, see Theorem 5, p. 183, of \cite{CG-Manusc}:

\begin{corollary}\label{W-CG}
Let $\Sigma^2$ be compact minimal surface of $\s^n(c)$ or $\h^n(c), \ n\geq3.$ If $\Sigma^2\subset\s^n(c)$ assume further that $\diam\Sigma\leq\frac{\pi}{2\sqrt{c}}.$ Let $A=\area(\Sigma)$ and $L=\lenght(\partial\Sigma).$ Then
\[
2\pi A\leq L^2 + cA^2.
\]
\end{corollary}

\section{Proof of the Main Results}\label{Main}

We start with a well known result, which we give a proof here for the sake of completeness. The proof presented here is essentially in \cite{Brendle}, Lemma 2.2, p. 253.

\begin{lemma}\label{lemma1}
Let $M^n=I\times N^{n-1},\ n\geq3,$ be a warped product manifold with metric $g=dr^2+h(r)^2g_N,$ where $g_N$ is the metric of $N^{n-1}$ and $r$ is the distance function of $M^n.$ Then
\begin{equation}\label{hessian}
\hess r(U,V) = \frac{h'(r)}{h(r)}[\lan U,V\ran - \lan \n r,U\ran \lan \n r,V \ran],
\end{equation}
for all $U,V\in TM.$ Moreover, if $\Sigma$ is a $k-$dimensional submanifold of $M^n,$ then 
\begin{equation}\label{hessian-2}
\Delta_\Sigma r = \dfrac{h'(r)}{h(r)}[k-|\n_\Sigma r|^2] + k\lan\vec{H},\n r\ran.
\end{equation}
Here $\lan U,V\ran = g(U,V),$ $\n r$ is the gradient of $r$ in $M^n,$ $\n_\Sigma r$ and $\Delta_\Sigma r$ denote the gradient and the Laplacian of $r$ in $\Sigma,$ respectively. 
\end{lemma}
\begin{proof}

Taking the Lie derivative of $g$ in the direction of $\partial_r = \n r,$ we have
\[
\begin{split}
\mathcal{L}_{\partial_r}(g)&=\mathcal{L}_{\partial_r}(dr\otimes dr) + \mathcal{L}_{\partial_r}(h(r)^2g_N)\\
&= 2 \mathcal{L}_{\partial_r}(dr)\otimes dr + \mathcal{L}_{\partial_r}(h(r)^2)g_N\\
&= 2d(\mathcal{L}_{\partial_r} r)\otimes dr + \lan\n(h(r)^2),\partial_r\ran g_N\\
&= 2 d(\lan\n r,\partial_r\ran)\otimes dr + 2h'(r)h(r)g_N\\
&= 2h'(r)h(r)g_N,\\
\end{split}
\]
where $\mathcal{L}_{\partial_r} g_N=0$ provided $g_N$ does not depends on $r.$ On the other hand,
\[
\begin{split}
(\mathcal{L}_{\partial_r}(g))(U,V) &= \lan U,\n_V \partial_r\ran + \lan\n_U\partial_r,V\ran\\
                                      &= 2\hess r(U,V).
\end{split}
\]
Since $g_N = \dfrac{1}{h(r)^2}(g - dr^2),$ we have
\[
\begin{split}
\hess r(U,V)&=\frac{h'(r)}{h(r)}[g(U,V) - dr^2(U,V)]\\
            &=\frac{h'(r)}{h(r)}[\lan U,V\ran - \lan U,\partial_r\ran \lan V,\partial_r\ran].
\end{split}
\]
The expression (\ref{hessian-2}) follows by tracing the known identity
\[
\begin{split}
\hess_\Sigma r(U,V) &= \hess r(U,V)+ \lan II(U,V),\n r\ran\\
                    &=\frac{h'(r)}{h(r)}[\lan U,V\ran - \lan \n r,U\ran \lan \n r,V \ran] + \lan II(U,V),\n r\ran
\end{split}
\]
over $\Sigma,$ where $\hess_\Sigma r$ denotes the Hessian of $r$ in $\Sigma$ and $II(U,V)$ is the second fundamental form of $\Sigma.$
\end{proof}

The next proposition will give the fundamental inequalities for the proof of Theorem \ref{iso-SS1} and Theorem \ref{iso-RN}.

\begin{proposition}\label{Teo-1}
Let $M^n=I\times N^{n-1}, n\geq 3,$ be a warped product manifold, with the warping function $h:I\ria\R$ satisfying $h(r)>0$ and $h'(r)\neq 0$ for all $r\in I$. If $\Sigma$ is a $k-$dimensional, compact, oriented submanifold of $M^n,$ possibly with boundary, then
\begin{equation}\label{eq-Teo-1-1}
\begin{split}
\int_\Sigma f d\Sigma &=\! \dfrac{1}{k}\left[\int_{\partial\Sigma} f\dfrac{h(r)}{h'(r)}\lan\n_\Sigma r,\nu\ran dS_\Sigma\! +\! \int_\Sigma \left(\lan -k\vec{H}, \n r\ran f - \lan \n_\Sigma f, \n_\Sigma r\ran\right) \frac{h(r)}{h'(r)}d\Sigma\right]\\
&\qquad \qquad -\frac{1}{k(n-1)}\int_\Sigma f\ric_M(\n r)\left(\frac{h(r)}{h'(r)}\right)^2|\n_\Sigma r|^2 d\Sigma,
\end{split}
\end{equation} 
for every non-negative smooth function $f:\Sigma\ria\R,$ where $\vec{H}$ denotes the mean curvature vector field of $\Sigma,$ $\nu$ is the unitary conormal vector field of $\partial\Sigma$ pointing outward, $\n_\Sigma r$ denotes the gradient of $r$ in $\Sigma$ and $\n r$ denotes the gradient of $r$ in $M^n.$ In particular, for $f\equiv 1$ and $h'(r)>0,$ $r\in I,$ we have
\begin{equation}\label{eq-Teo-1-2}
\begin{split}
|\Sigma|&\leq \frac{1}{k}\left[\int_{\partial\Sigma} \frac{h(r)}{h'(r)}dS_\Sigma + \int_\Sigma \lan -k\vec{H}, \n r\ran \frac{h(r)}{h'(r)}d\Sigma\right]\\
&\qquad -\dfrac{1}{k(n-1)}\int_\Sigma \ric_M(\n r)\left(\frac{h(r)}{h'(r)}\right)^2|\n_\Sigma r|^2 d\Sigma.
\end{split}
\end{equation} 
Moreover, if $\Sigma=\{r\}\times N^{n-1}$ is a slice, or $\Sigma$ is a totally geodesic submanifold of dimension $k,$ then the equality in (\ref{eq-Teo-1-2}) holds.
\end{proposition}

\begin{proof}
Let $u:I\ria \R$ be a real function such that $u'(r)=h(r).$ By tracing the expression
\[
\hess_\Sigma u(r)(U,V) = \hess_M u(r)(U,V) + \lan II(U,V),\n u(r)\ran
\]
over $\Sigma,$ we have
\[
\Delta_\Sigma u(r) = \sum_{i=1}^k \hess_M u(r)(e_i,e_i) + k\lan\vec{H},\n u(r)\ran
\]
for any orthonormal frame $\{e_1,e_2,\ldots,e_k\}$ of $\Sigma.$ On the other hand, since
\[
\Delta_\Sigma u(r)  = \di_\Sigma(\n_\Sigma u(r))= \di_\Sigma( h(r)\n_\Sigma r)
\]
and, by using Lemma \ref{lemma1},
\[
\begin{split}
\sum_{i=1}^k \hess_M u(r)(e_i,e_i)&=\sum_{i=1}^k \lan\n_{e_i}\n u(r),e_i\ran = \sum_{i=1}^k\lan\n_{e_i}(h(r)\n r),e_i\ran\\
                               &=h'(r)|\n_\Sigma r|^2 + h(r)\sum_{i=1}^k \hess_M r (e_i,e_i)\\
                               &=h'(r)|\n_\Sigma r|^2 + h(r)\sum_{i=1}^k \dfrac{h'(r)}{h(r)}[\lan e_i,e_i\ran - \lan e_i,\n r\ran^2]\\
                               &=kh'(r),
\end{split}
\]
we obtain
\begin{equation}\label{eq-div}
\di_\Sigma( h(r)\n_\Sigma r) = kh'(r) + kh(r)\lan\vec{H},\n r\ran.
\end{equation}
Let $f:\Sigma\ria\R$ be a smooth function. By using the equation (\ref{eq-div}), we have
\[
\begin{split}
\di_\Sigma\left(\frac{f}{h'(r)}h(r)\n_\Sigma r\right) &= \left\lan\n_\Sigma \left(\frac{f}{h'(r)}\right), h(r)\n_\Sigma r\right\ran + \frac{f}{h'(r)}\di_\Sigma(h(r)\n_\Sigma r)\\
                          &=\frac{h(r)}{h'(r)}\lan\n_\Sigma f, \n_\Sigma r \ran - f\frac{h(r)h''(r)}{h'(r)^2}|\n_\Sigma r|^2\\
                          &\qquad + kf + kf\frac{h(r)}{h'(r)}\lan \vec{H},\n r \ran\\
                          &=\frac{h(r)}{h'(r)}\lan\n_\Sigma f, \n_\Sigma r \ran - f\frac{h''(r)}{h(r)}\left(\frac{h(r)}{h'(r)}\right)^2|\n_\Sigma r|^2\\
                          &\qquad + kf + kf\frac{h(r)}{h'(r)}\lan \vec{H},\n r \ran.\\
\end{split}
\]
Integrating the expression above over $\Sigma$ and using the divergence theorem, we have 
\begin{equation}\label{proof-prop-1}
\begin{split}
k\int_\Sigma f d\Sigma &=\int_\Sigma \di_\Sigma\left(f\frac{h(r)}{h'(r)}\n_\Sigma r\right)d\Sigma -\int_\Sigma \frac{h(r)}{h'(r)} \lan \n_\Sigma f,\n_\Sigma r\ran d\Sigma\\
                        &\qquad + \int_\Sigma f\frac{h''(r)}{h(r)}\left(\frac{h(r)}{h'(r)}\right)^2|\n_\Sigma r|^2 d\Sigma - k\int_\Sigma f\frac{h(r)}{h'(r)} \lan \vec{H},\n r\ran d\Sigma \\
                        & =\int_{\partial\Sigma} f\frac{h(r)}{h'(r)}\lan \n_\Sigma r, \nu\ran dS_\Sigma -\int_\Sigma \frac{h(r)}{h'(r)} \lan \n_\Sigma f,\n_\Sigma r\ran d\Sigma\\
                        &\qquad + \int_\Sigma f\frac{h''(r)}{h(r)}\left(\frac{h(r)}{h'(r)}\right)^2|\n_\Sigma r|^2 d\Sigma - k\int_\Sigma f\frac{h(r)}{h'(r)} \lan \vec{H},\n r\ran d\Sigma,\\
\end{split}
\end{equation}
where $\nu$ is the unitary conormal vector field of $\partial\Sigma$ pointing outward. The identity (\ref{eq-Teo-1-1}) follows from (\ref{proof-prop-1}) by noting that
\[
\ric_M(\n r) = -(n-1)\dfrac{h''(r)}{h(r)}.
\]
The inequality (\ref{eq-Teo-1-2}) follows considering $f\equiv 1$ and observing that
\[
\lan \n_\Sigma r, \nu\ran \leq |\n_\Sigma r||\nu| \leq |\n r||\nu|=1.
\]
To see the cases of the equality in the inequality (\ref{eq-Teo-1-2}), notice that, if $\Sigma=\{r_0\}\times N^{n-1}$ is a slice, then $\partial\Sigma =\emptyset,$ $\n_\Sigma r =0,$ $h(r)=h(r_0), \ h'(r)=h'(r_0)$ and $\vec{H}=-\dfrac{h'(r_0)}{h(r_0)}\n r.$ If $\Sigma$ is totally geodesic, then $\n_\Sigma r=\n r=\nu$ and $\vec{H}=0.$ This implies that the inequalities in the proof of Proposition \ref{Teo-1} become equalities.
\end{proof}

\begin{example}
{\normalfont
Even in the simplest case of surfaces in $\R^3$ we can find classes of examples different from the slices (i.e., the round spheres centered in the origin) which satisfy the equality in the inequality (\ref{eq-Teo-1-2}) of Proposition \ref{Teo-1}. In fact, consider the right cones with central angle $2\alpha$ parametrized in spherical coordinates by $F^\alpha:(0,2\pi)\times (0,R)\ria \R^3,$
\[
F^\alpha(\theta,r)=(r\sin\alpha\cos\theta,r\sin\alpha\sin\theta,r\cos\alpha).
\]
If $\Sigma^\alpha=F^\alpha((0,2\pi)\times (0,R)),$ then 
\[
|\Sigma^\alpha|=\int_0^{2\pi}\int_0^R\|F^\alpha_r\times F^\alpha_\theta\|dr\ d\theta = \pi R^2\sin\alpha,
\]
$|\partial\Sigma^\alpha|=2\pi R\sin\alpha,$ $\lan\vec{H},\n r\ran=0,$ and, since $\partial\Sigma^\alpha$ lies in the sphere of radius $R$ we have
\[
|\Sigma^\alpha|=\pi R^2\sin\alpha =\dfrac{R}{2}|\partial\Sigma^\alpha|= \frac{1}{2}\int_{\partial\Sigma^\alpha} r dS_{\Sigma^\alpha} = \frac{1}{2}\int_{\partial\Sigma^\alpha} \dfrac{h(r)}{h'(r)} dS_{\Sigma^\alpha},
\]
which is the equality in the inequality (\ref{eq-Teo-1-2}), since $\R^3$ is Ricci flat.
}
\end{example}

When $\Sigma$ is compact without boundary, the identity (\ref{eq-div}) in the proof of Proposition \ref{Teo-1} gives rise to the following Hsiung-Minkowski type identity:

\begin{corollary}\label{cor-mink}
Let $M^n=I\times N^{n-1}, n\geq 3,$ be a warped product manifold, with the warping function $h:I\ria\R$ satisfying $h(r)>0$ and $h'(r)>0$ for all $r\in I$. If $\Sigma$ is a $k-$dimensional, compact, without boundary, oriented submanifold of $M^n,$ then
\begin{equation}\label{eq-mink}
\int_{\Sigma} [h'(r) + h(r)\lan\vec{H},\n r\ran]d\Sigma =0.
\end{equation}
In particular, there is no compact, without boundary, minimal submanifolds in $M^n.$
\end{corollary}

\begin{remark}
{\normalfont
The identity (\ref{eq-mink}) can be compared with the known Hsiung-Minkow-\\ski inequalities of \cite{hsiung}, \cite{reilly} and \cite{heintze} for compact hypersurfaces in the space forms $\R^n,$ $\s^n(c)$ and $\h^n(c).$
}
\end{remark}

\begin{remark}
{\normalfont
It is a classical result that there is no compact, without boundary, minimal surfaces in $\R^3.$ This result was generalized by S. Myers in \cite{myers}, who proved the non-existence of compact, without boundary, minimal hypersurfaces in simply connected Riemannian manifolds of non positive sectional curvature and in the open hemisphere of $\s^n(c).$ For complete, non-compact, Riemannian manifolds of positive sectional curvature, the non-existence result was proved by K. Shiorama, see \cite{shiorama}, and for Riemannian manifolds of sectional curvature bounded above by a constant, the non-existence of compact, without boundary, minimal submanifolds was proved by J. Lu and M. Tanaka, see \cite{LT}.
}
\end{remark}

In the following, we prove Theorem \ref{iso-SS1} and Theorem \ref{iso-RN}.

\begin{proof}[Proof of Theorem \ref{iso-SS1}.] We start with the expression (\ref{defi-SS}). This expression is equivalent to
\begin{equation}\label{defi-SS2}
h'(r)^2 = 1-mh(r)^{2-n} - ch(r)^2.
\end{equation}
By taking implicit derivatives in (\ref{defi-SS2}), we have
\[
2h'(r)h''(r) = m(n-2)h(r)^{1-n}h'(r) - 2ch(r)h'(r).
\]
Therefore,
\begin{equation}\label{SS3}
-\dfrac{1}{(n-1)}\ric_{M^n(c)}(\n r) = \frac{h''(r)}{h(r)} = \dfrac{m(n-2)}{2h(r)^n} - c,
\end{equation}
since $h'(r)>0.$ Replacing the estimate
\[
\lan-\vec{H},\n r\ran \leq |\vec{H}||\n r| = |\vec{H}|,
\]
in the isoperimetric inequality (\ref{eq-Teo-1-2}), we obtain
\begin{equation}\label{iso-SS-prov}
\begin{split}
|\Sigma|&\leq \frac{1}{k}\left[\int_{\partial\Sigma}\dfrac{h(r)}{h'(r)}d\Sigma + k\int_\Sigma |\vec{H}|\dfrac{h(r)}{h'(r)}d\Sigma\right]\\
&\qquad -\frac{1}{k(n-1)} \int_\Sigma\ric_{M^n(c)}(\n r)\left(\frac{h(r)}{h'(r)}\right)^2|\n_\Sigma r|^2d\Sigma.
\end{split}
\end{equation}
In order to conclude the proof of Theorem \ref{iso-SS1}, we need to estimate the function $\dfrac{h(r)}{h'(r)},$ and for that we will analyse when $\dfrac{h(r)}{h'(r)}$ is increasing or decreasing. By using (\ref{defi-SS2}) and (\ref{SS3}), we have
\begin{equation}\label{SS4}
\begin{split}
\dfrac{d}{dr}\left(\dfrac{h(r)}{h'(r)}\right)&=\frac{h'(r)^2 - h(r)h''(r)}{h'(r)^2}\\
                                       &=\dfrac{1-mh(r)^{2-n} - ch(r)^2 - \frac{m(n-2)}{2}h(r)^{2-n}+ch(r)^2}{h'(r)^2}\\
									   &=\dfrac{1-\frac{mn}{2}h(r)^{2-n}}{h'(r)^2}.\\
\end{split}
\end{equation}

This implies that $\dfrac{h(r)}{h'(r)}$ is decreasing for $h(r)\in \left(s_0,\left(\frac{mn}{2}\right)^{\frac{1}{n-2}}\right)$ and that $\dfrac{h(r)}{h'(r)}$ is increasing for $h(r)\in \left(\left(\frac{mn}{2}\right)^{\frac{1}{n-2}},s_1\right).$

To estimate the third integral of (\ref{iso-SS-prov}) we need to know the sign of the Ricci curvature. By using (\ref{SS3}), if $c\leq 0,$ then $-\frac{1}{n-1}\ric_{M^n(c)}(\n r)>0$ everywhere in $(s_0,\infty)$, and if $c>0,$ then it happens for $h(r)<\left(\frac{m(n-2)}{2c}\right)^{\frac{1}{n}}.$ Notice also that the condition $\frac{n^n}{4(n-2)^{n-2}}m^2c^{n-2}<1$ is equivalent to $\left(\frac{m(n-2)}{2c}\right)^{\frac{1}{n}}>\left(\frac{mn}{2}\right)^{\frac{1}{n-2}},$  i.e., $\left(\frac{m(n-2)}{2c}\right)^{\frac{1}{n}}\in \left(\left(\frac{mn}{2}\right)^{\frac{1}{n-2}},s_1\right).$

 Let 
\[
d_\Sigma = \min\{s\in (s_0,s_1)| \Sigma\cap\{\{s\}\times\s^{n-1}\}\neq\emptyset\}
\] 
and 
\[
R_\Sigma = \max\{s\in (s_0,s_1)| \Sigma\cap\{\{s\}\times\s^{n-1}\}\neq\emptyset\}.
\]
Since $\dfrac{h(r)}{h'(r)}$ is decreasing for $s=h(r)\in \left(s_0,\left(\frac{mn}{2}\right)^{\frac{1}{n-2}}\right),$ we have
\begin{equation}\label{est-SS1-1}
\frac{h(r)}{h'(r)}=\dfrac{s}{\sqrt{1-ms^{2-n}-cs^2}}\leq \frac{d_\Sigma}{\sqrt{1-md_\Sigma^{2-n}-cd_\Sigma^2}}
\end{equation}
for $s\!\in\! \left(\!s_0,\!\left(\frac{mn}{2}\right)^{\frac{1}{n-2}}\!\right).$ Since $-\frac{1}{n-1}\ric_{M^n(c)}(\n r)>0$ for $s\!=\!h(r)\!\in\!\left(\!s_0,\!\left(\frac{mn}{2}\right)^{\frac{1}{n-2}}\!\right)\!,$ by using (\ref{est-SS1-1}) into (\ref{iso-SS-prov}), we obtain the inequality (\ref{iso-SS1-d}).

In order to prove (\ref{iso-SS1-mod-3}), notice that
\[
-\frac{1}{n-1}\ric_{M^n(c)}(\n r)\left(\dfrac{h(r)}{h'(r)}\right)^2 = \left(\frac{m(n-2)}{2h(r)^n}-c\right)\frac{h(r)^2}{1-mh(r)^{2-n}-ch(r)^2}.
\]
Let $f:(s_0,s_1)\ria\R$ defined by
\begin{equation}\label{f}
f(t)=\left(\frac{m(n-2)}{2t^n}-c\right)\frac{t^2}{1-mt^{2-n}-ct^2}=\dfrac{\frac{1}{2}m(n-2)t^{2-n}-ct^2}{1-mt^{2-n}-ct^2}.
\end{equation}
Since $f$ is a product of two decreasing functions in the interval $\left(s_0,\left(\frac{mn}{2}\right)^{\frac{1}{n-2}}\right),$ we have
\[
f(h(r))\leq f(d_\Sigma).
\]
This and the fact $|\n_\Sigma r|\leq 1$ imply
\[
\begin{split}
|\Sigma|&\leq \frac{d_\Sigma}{k\sqrt{1-md_\Sigma^{2-n}-cd_\Sigma^2}}\left[|\partial \Sigma|+k\int_\Sigma|\vec{H}|d\Sigma\right]\\
&\qquad+ \frac{1}{k}\int_\Sigma\frac{-1}{n-1}\ric_{M^n(c)}(\n r)\left(\dfrac{h(r)}{h'(r)}\right)^2|\n_\Sigma r|^2d\Sigma\\
&\leq \frac{d_\Sigma}{k\sqrt{1-md_\Sigma^{2-n}-cd_\Sigma^2}}\left[|\partial \Sigma|+k\int_\Sigma|\vec{H}|d\Sigma\right] + \dfrac{f(d_\Sigma)}{k}|\Sigma|,\\
\end{split}
\]
i.e.,
\[
|\Sigma|\leq \frac{d_\Sigma}{(k-f(d_\Sigma))\sqrt{1-md_\Sigma^{2-n}-cd_\Sigma^2}}\left[|\partial \Sigma|+k\int_\Sigma|\vec{H}|d\Sigma\right].
\]
Since
\[
\begin{split}
k-f(d_\Sigma)&= k - \dfrac{\frac{1}{2}m(n-2)d_\Sigma^{2-n}-cd_\Sigma^2}{1-md_\Sigma^{2-n}-cd_\Sigma^2}\\
&=\dfrac{k-\frac{m}{2}(n+2k-2)d_\Sigma^{2-n}-c(k-1)d_\Sigma^2}{1-md_\Sigma^{2-n}-cd_\Sigma^2}\\
&=\dfrac{\left(1-\frac{mn}{2}d_\Sigma^{2-n}\right) + (k-1)(1-md_\Sigma^{2-n}-cd_\Sigma^2)}{1-md_\Sigma^{2-n}-cd_\Sigma^2},
\end{split}
\]
the inequality (\ref{iso-SS1-mod-3}) follows. This concludes the proof of the item (i) of Theorem \ref{iso-SS1}.

Now, we prove the item (ii) of Theorem \ref{iso-SS1}. Notice that, since $\dfrac{h(r)}{h'(r)}$ is increasing for $s=h(r)\in\left(\left(\frac{mn}{2}\right)^{\frac{1}{n-2}},s_1\right),$ we have
\begin{equation}\label{est-SS1-2}
\frac{h(r)}{h'(r)}=\dfrac{s}{\sqrt{1-ms^{2-n}-cs^2}}\leq \frac{R_\Sigma}{\sqrt{1-mR_\Sigma^{2-n}-cR_\Sigma^2}}.
\end{equation}
Since, if $c>0,$ $-\frac{1}{n-1}\ric(\n r)<0$ for $s=h(r)\in\left(\left(\frac{m(n-2)}{2c}\right)^{\frac{1}{n}},s_1\right),$ by using 
\begin{equation}\label{est-SS1-31}
\frac{h(r)}{h'(r)}=\dfrac{s}{\sqrt{1-ms^{2-n}-cs^2}}\geq \frac{d_\Sigma}{\sqrt{1-md_\Sigma^{2-n}-cd_\Sigma^2}},
\end{equation}
for $s=h(r)\in\left(\left(\frac{mn}{2}\right)^{\frac{1}{n-2}},s_1\right),$  and (\ref{est-SS1-2}) in (\ref{iso-SS-prov}), we obtain (\ref{iso-SS1-R}). This proves the item (ii) of Theorem \ref{iso-SS1}.

Let us prove the item (iii) of Theorem \ref{iso-SS1}. Since $-\frac{1}{n-1}\ric(\n r)>0$ for $s=h(r)\in\left(\left(\frac{mn}{2}\right)^{\frac{1}{n-2}},\infty\right),$ when $c\leq 0,$ and for  $s=h(r)\in\left(\left(\frac{mn}{2}\right)^{\frac{1}{n-2}},\left(\frac{m(n-2)}{2c}\right)^{\frac{1}{n}}\right),$ when $c>0,$ then by using (\ref{est-SS1-2}) in (\ref{iso-SS-prov}) we obtain (\ref{iso-SS1-R}). This proves the item (iii) of Theorem \ref{iso-SS1}.

Let us prove the item (iv). In order to prove (\ref{iso-SS1-mod}), notice that the function $f(t)$ defined in (\ref{f}) satisfies $f(t)<1$ for $t\in\left(\left(\frac{mn}{2}\right)^{\frac{1}{n-2}},s_1\right).$ This implies
\[
\begin{split}
|\Sigma|&\leq \frac{R_\Sigma}{\sqrt{1-mR_\Sigma^{2-n}-cR_\Sigma^2}}\left[|\partial \Sigma|+k\int_\Sigma|\vec{H}|d\Sigma\right]\\
&\qquad + \frac{1}{k}\int_\Sigma\frac{-1}{n-1}\ric_{M^n(c)}(\n r)\left(\dfrac{h(r)}{h'(r)}\right)^2|\n_\Sigma r|^2d\Sigma\\
&\leq \frac{R_\Sigma}{\sqrt{1-mR_\Sigma^{2-n}-cR_\Sigma^2}}\left[|\partial \Sigma|+k\int_\Sigma|\vec{H}|d\Sigma\right] + \frac{1}{k}|\Sigma|,\\
\end{split}
\]
provided $|\n_\Sigma r|\leq 1.$ Therefore, the inequality (\ref{iso-SS1-mod}) follows immediately. The inequality (\ref{iso-SS1-mod-2}) follows observing that 
\[
\frac{R_\Sigma}{\sqrt{1-mR_\Sigma^{2-n}-cR_\Sigma^2}}<\frac{1}{\sqrt{-c}}
\]
for $c<0.$ This concludes the proof of the item (iv) and the proof of Theorem \ref{iso-SS1}.

\end{proof}

\begin{proof}[Proof of Corollary \ref{iso-Hn}.]
Taking $m\ria 0$ in the equation (\ref{defi-SS}), for $c<0,$ we have 
\[
h'(r)=\sqrt{1-ch(r)^2}, \ h(0)=0,
\]
with solution $h(r)=\dfrac{1}{\sqrt{-c}}\sinh(\sqrt{-c}r).$ Replacing $\dfrac{h(r)}{h'(r)}=\dfrac{\tanh(\sqrt{-c}r)}{\sqrt{-c}}$ and \[\ric_{\h^n(c)}(\n r) =-(n-1)\frac{h''(r)}{h(r)}=(n-1)c,\] in the inequality (\ref{iso-SS-prov}) in the proof of Theorem \ref{iso-SS1}, for $c<0$, and by using that $\tanh(\sqrt{-c}r)<\tanh(\sqrt{-c}\widetilde{R_\Sigma})$ in $\Sigma,$ we have the result.
\end{proof}

\begin{proof}[Proof of Corollary \ref{iso-Sn}.]
Taking $m\ria 0$ in the equation (\ref{defi-SS}), for $c>0,$ we have 
\[
h'(r)=\sqrt{1-ch(r)^2}, \ h(0)=0,
\]
with solution $h(r)=\dfrac{1}{\sqrt{c}}\sin(\sqrt{c}r).$ Replacing $\dfrac{h(r)}{h'(r)}=\dfrac{\tan(\sqrt{c}r)}{\sqrt{c}},$ and \[\ric_{\s^n_+(c)}(\n r) =-(n-1)\frac{h''(r)}{h(r)}=(n-1)c,\] 
in the inequality (\ref{iso-SS-prov}) in the proof of Theorem \ref{iso-SS1}, for $c>0$, and by using that $\tan(\sqrt{c}r)<\tan(\sqrt{c}\widetilde{R_\Sigma})$ in $\Sigma,$ we have the result.
\end{proof}

Now we prove Theorem \ref{iso-RN}.

\begin{proof}[Proof of Theorem \ref{iso-RN}.]\label{proof-iso-RN} The identity (\ref{defi-RN}) is equivalent to
\begin{equation}\label{RN-1}
h'(r)^2 = 1-mh(r)^{2-n} + q^2h(r)^{4-2n}.
\end{equation}
By taking implicit derivatives in (\ref{RN-1}), we have
\[
2h'(r)h''(r) = m(n-2)h(r)^{1-n}h'(r) - 2(n-2)q^2h(r)^{3-2n}h'(r),
\]
which gives
\[
\dfrac{h''(r)}{h(r)} = \dfrac{m(n-2)}{2}h(r)^{-n} - (n-2)q^2h(r)^{2-2n} = \dfrac{n-2}{2h(r)^n}\left(m-\dfrac{2q^2}{h(r)^{n-2}}\right),
\]
i.e.,
\begin{equation}\label{RN-2}
-\frac{1}{n-1}\ric_{M^n}(\n r)=\dfrac{h''(r)}{h(r)} = \dfrac{n-2}{2h(r)^n}\left(m-\dfrac{2q^2}{h(r)^{n-2}}\right),
\end{equation}
provided $h'(r)\!>\!0.$ We have $-\frac{1}{n-1}\!\ric_{M^n}\!(\n r)\!\! >\! 0$ if, and only if, $s\! =\! h(r)\!\! >\!\!\left(\!\frac{2q^2}{m}\!\right)^{\frac{1}{n-2}}\! .$ Since $\left(\frac{2q^2}{m}\right)^{\frac{1}{n-2}}\!\!<\!\left(\!\frac{2q^2}{m-\sqrt{m^2-4q^2}}\!\right)^{\frac{1}{n-2}}\!\!=\!s_0,$ we have $-\frac{1}{n-1}\!\ric_{M^n}\!(\n r)\!\!>\!0$ everywhere in the Reissner-Nordstrom manifold.

Replacing the estimate
\[
\lan -\vec{H},\n r\ran \leq |\vec{H}|,
\]
into the isoperimetric inequality (\ref{eq-Teo-1-2}) of Proposition \ref{Teo-1}, we have
\begin{equation}\label{RN-3}
\begin{split}
|\Sigma|&\leq \dfrac{1}{k}\left[\int_\Sigma \dfrac{h(r)}{h'(r)}d\Sigma + \int_\Sigma |\vec{H}|\dfrac{h(r)}{h'(r)}d\Sigma\right]\\
&\qquad -\dfrac{1}{k(n-1)}\int_\Sigma\ric_{M^n}(\n r)\left(\frac{h(r)}{h'(r)}\right)^2|\n_\Sigma r|^2d\Sigma.\\
\end{split}
\end{equation}

In order to conclude the proof of Theorem \ref{iso-RN}, we need to estimate the quotient $\dfrac{h(r)}{h'(r)},$ and for that we will analyse when $\dfrac{h(r)}{h'(r)}$ is increasing or decreasing. By using (\ref{RN-1}) and (\ref{RN-2}), we obtain
\[
\begin{split}
\dfrac{d}{dr}\left(\dfrac{h(r)}{h'(r)}\right)&=\frac{h'(r)^2 - h(r)h''(r)}{h'(r)^2} \\
&= \dfrac{1-mh(r)^{2-n} + q^2 h(r)^{4-2n} - \frac{m(n-2)}{2}h(r)^{2-n} + (n-2)q^2h(r)^{4-2n}}{h'(r)^2}\\
&=\dfrac{1-\frac{mn}{2}h(r)^{2-n} + (n-1)q^2h(r)^{4-2n}}{h'(r)^2}.\\
\end{split}
\]
Notice that the expression $1-\frac{mn}{2}h(r)^{2-n} + (n-1)q^2h(r)^{4-2n}$ is a quadratic function of $u=h(r)^{2-n}.$ Let
\[
P(u)=1-\dfrac{mn}{2}u + (n-1)q^2u^2,
\]
which has the two different roots
\[
\alpha_1:=\dfrac{mn - \sqrt{m^2n^2-16(n-1)q^2}}{4(n-1)q^2} \ \mbox{and} \ \alpha_2:=\dfrac{mn + \sqrt{m^2n^2-16(n-1)q^2}}{4(n-1)q^2}.
\]
Since $P(u)>0$ for $u<\alpha_1$ and for $u>\alpha_2,$ and $u=h(r)^{2-n},$ we have that $\dfrac{d}{dr}\left(\dfrac{h(r)}{h'(r)}\right)>0$ for 
\[
h(r)>\alpha_1^{-\frac{1}{n-2}}:=s_2 \ \mbox{and} \ h(r)<\alpha_2^{-\frac{1}{n-2}}:=s_3.
\]
Now, we need to see if $s_2\in(s_0,\infty)$ and $s_3\in(s_0,\infty).$ Define $Q(u)=1-mu+q^2u^2.$ Since
\[
Q(u)-P(u) = \frac{(n-2)}{2}u(m-2q^2u^2),
\]
we have $Q(u)>P(u)$ for $0<u<\dfrac{m}{2q^2}.$ If we denote by $\beta_1<\beta_2$ the roots of $Q(u),$ we have also \[\beta_1<\frac{m}{2q^2}<\beta_2.\]
These facts imply that $\alpha_1<\beta_1.$ Since $P(u)>Q(u)$ for $u>\dfrac{m}{2q^2}$ and $P\left(\dfrac{m}{2q^2}\right)=Q\left(\dfrac{m}{2q^2}\right)=1-\dfrac{m^2}{2q^2}<0,$ we have $\alpha_2<\beta_2,$ i.e.,
\[
\alpha_1<\beta_1<\alpha_2<\beta_2.
\]
Since $\alpha_1=s_2^{2-n}, \ \beta_1=s_0^{2-n}$ and $\alpha_2=s_3^{2-n},$ we have
\[
s_3<s_0<s_2,
\]
i.e., $s_2\in(s_0,\infty)$ and $s_3\not\in(s_0,\infty).$ Therefore
\[
\dfrac{d}{dr}\left(\frac{h(r)}{h'(r)}\right)<0 \ \mbox{for} \ h(r)\in(s_0,s_2) \ \mbox{and} \ \dfrac{d}{dr}\left(\frac{h(r)}{h'(r)}\right)>0 \ \mbox{for} \ h(r)\in(s_2,\infty).
\]
Since $\Sigma$ is compact, we can consider
\[
d_\Sigma = \min\{s\in (s_0,\infty)| \Sigma\cap\{\{s\}\times\s^{n-1}\}\neq\emptyset\}
\] 
and 
\[
R_\Sigma = \max\{s\in (s_0,\infty)| \Sigma\cap\{\{s\}\times\s^{n-1}\}\neq\emptyset\}.
\]
We have
\begin{equation}\label{RN-est-1}
\dfrac{h(r)}{h'(r)}=\dfrac{s}{\sqrt{1-ms^{2-n}+q^2s^{4-2n}}}\leq \dfrac{d_\Sigma}{\sqrt{1-md_\Sigma^{2-n}+q^2d_\Sigma^{4-2n}}} \ \mbox{for}\ s\in (s_0,s_2)
\end{equation}
and
\begin{equation}\label{RN-est-2}
\dfrac{h(r)}{h'(r)}=\dfrac{s}{\sqrt{1-ms^{2-n}+q^2s^{4-2n}}}\leq \dfrac{R_\Sigma}{\sqrt{1-mR_\Sigma^{2-n}+q^2R_\Sigma^{4-2n}}} \ \mbox{for}\ s\in (s_2,\infty).
\end{equation}
Replacing these estimates into (\ref{RN-3}) we obtain the inequalities (\ref{iso-RN-d1}) and (\ref{iso-RN-R}). 

On the other hand,
\[
\begin{split}
-\frac{1}{n-1}\ric_{M^n}(\n r)&\left(\dfrac{h(r)}{h'(r)}\right)^2=\frac{n-2}{2h(r)^{n}}\left(m-\frac{2q^2}{h(r)^{n-2}}\right)\left(\dfrac{h(r)}{h'(r)}\right)^2\\
&=\frac{n-2}{2h(r)^{n-2}}\left(m-\frac{2q^2}{h(r)^{n-2}}\right)\dfrac{1}{1-mh(r)^{2-n}+q^2h(r)^{4-2n}}\\
& = f(h(r)^{2-n}),
\end{split}
\]
where 
\[
f(t)=\dfrac{(n-2)(mt -2q^2t^2)}{2(1-mt+q^2t^2)}.
\] 
Since 
\[
f'(t)=\dfrac{(n-2)(mq^2t^2 - 4q^2t + m)}{2(1-mt+q^2t^2)^2}
\]
and $m>2q,$ we have that $f(t)$ is increasing for every $t.$ This implies that $\tilde{f}(r) = f(h(r)^{2-n})$ is decreasing for every $r>0$ and thus 
\begin{equation}\label{RN-est-3}
-\frac{1}{n-1}\ric_{M^n}(\n r)\left(\dfrac{h(r)}{h'(r)}\right)^2 = f (h(r)^{2-n})\leq f(d_\Sigma^{2-n})= C_3(d_\Sigma).
\end{equation}
Replacing the estimates (\ref{RN-est-1}) and (\ref{RN-est-3}) in (\ref{RN-3}), and by using that $|\n_\Sigma r|\leq 1,$ we obtain
\[
|\Sigma|\leq \dfrac{d_\Sigma}{k\sqrt{1-md_\Sigma^{2-n}+q^2d_\Sigma^{4-2n}}}\left[|\partial\Sigma| +\int_\Sigma |\vec{H}|d\Sigma \right] + \frac{C_3(d_\Sigma)}{k}|\Sigma|,
\]
and then the inequality (\ref{iso-RN-mod-1}) follows. Analogously, replacing the estimates (\ref{RN-est-2}) and (\ref{RN-est-3}) in (\ref{RN-3}), we obtain the inequality (\ref{iso-RN-mod-3}). This concludes the proof of Theorem \ref{iso-RN}.
\end{proof}

\begin{proof}[Proof of Theorem \ref{mono}.] Let $\lambda:\R\ria\R$ be a smooth function such that $\lambda(t)=0$ for $t\leq0$ and $\lambda'(t)>0$ for $t>0.$ By using (\ref{eq-div}), we have
\begin{equation}\label{eq-div-lambda}
\begin{split}
\di_\Sigma(\lambda(R-r)h(r)\n_\Sigma r)&= h(r) \lan\n_\Sigma r,\n_\Sigma (\lambda(R-r))\ran + \lambda(R-r)\di_\Sigma(h(r)\n_\Sigma r)\\
									&= -\lambda'(R-r)h(r)|\n_\Sigma r|^2 + k\lambda(R-r)h'(r)\\
									&\qquad + k\lambda(R-r)h(r)\lan\vec{H},\n r\ran\\
\end{split}           
\end{equation}
for each $R>0.$ Since $\Sigma$ is proper, then $\lambda(R-r)h(r)\n_\Sigma r$ has compact support in $\Sigma\cap B_R.$ Thus, by using the divergence theorem in (\ref{eq-div-lambda}), we have
\begin{equation}\label{integral}
\int_\Sigma \lambda'(R-r)h(r)|\n_\Sigma r|^2 d\Sigma  = k\int_\Sigma \lambda(R-r)h'(r)d\Sigma + k\int_\Sigma \lambda(R-r)h(r)\lan\vec{H},\n r\ran d\Sigma.\\
\end{equation} 
From now on, we will continue the proof of the items (i) and (ii) separately.

\emph{Conclusion of the proof of the item (i).} Replacing the estimate 
\[
\int_\Sigma \lambda'(R-r)h(r)|\n_\Sigma r|^2 d\Sigma \leq \int_\Sigma \lambda'(R-r)h(r) d\Sigma = \dfrac{d}{dR}\left(\int_\Sigma \lambda(R-r)h(r) d\Sigma\right),\\
\]
in (\ref{integral}), by using the hypothesis that $\dfrac{h(r)}{h'(r)}$ is non-decreasing (i.e., $\dfrac{h'(r)}{h(r)}$ is non-increasing), and that $\lambda(R-r)=0$ for $r>R,$ we obtain
\[
\begin{split}
\dfrac{d}{dR}\left(\int_\Sigma\lambda(R-r)h(r)d\Sigma\right)&\geq k\dfrac{h'(R)}{h(R)}\int_\Sigma\lambda(R-r)h(r)d\Sigma\\
&\qquad + k\int_\Sigma \lambda(R-r)h(r)\lan\vec{H},\n r\ran d\Sigma.
\end{split}
\]
By using the fact
\[
\begin{split}
h(R)^k\dfrac{d}{dR}\left(\dfrac{1}{h(R)^k}\int_\Sigma \lambda(R-r)h(r) d\Sigma\right) &= \dfrac{d}{dR}\left(\int_\Sigma \lambda(R-r)h(r) d\Sigma\right)\\
&\qquad -\dfrac{kh'(R)}{h(R)}\int_\Sigma \lambda(R-r)h(r) d\Sigma,
\end{split}
\]
we have
\[
\dfrac{d}{dR}\left(\dfrac{1}{h(R)^k}\int_\Sigma \lambda(R-r)h(r) d\Sigma\right) \geq \dfrac{k}{h(R)^k}\int_\Sigma\lambda(R-r)h(r)\lan\vec{H},\n r\ran d\Sigma.\\
\]
Considering a sequence of functions $\lambda(t)$ converging to the characteristic function of the interval $[0,\infty),$ we obtain
\begin{equation}\label{mono-partial}
\dfrac{d}{dR}\left(\dfrac{1}{h(R)^k}\int_{\Sigma\cap B_R} h(r) d\Sigma\right) \geq \dfrac{k}{h(R)^k}\int_{\Sigma\cap B_R}h(r)\lan\vec{H},\n r\ran d\Sigma.
\end{equation}
By using the hypothesis $k|\vec{H}|\leq \alpha,$ we have
\[
\dfrac{d}{dR}\left(\dfrac{1}{h(R)^k}\int_{\Sigma\cap B_R} h(r) d\Sigma\right) \geq -\alpha\dfrac{1}{h(R)^k}\int_{\Sigma\cap B_R}h(r) d\Sigma,
\]
i.e.,
\begin{equation}\label{mono-final-1}
\dfrac{d}{dR}\log\left(\dfrac{1}{h(R)^k}\int_{\Sigma\cap B_R} h(r) d\Sigma\right) \geq -\alpha.
\end{equation}
Now, let $r_0,r_1\in I$ such that $r_0<r_1.$ Integrating (\ref{mono-final-1}) from $r_0$ to $r_1,$ we obtain
\[
\dfrac{e^{\alpha r_0}}{h(r_0)^k}\int_{\Sigma\cap B_{r_0}} h(s) d\Sigma \leq
\dfrac{e^{\alpha r_1}}{h(r_1)^k}\int_{\Sigma\cap B_{r_1}} h(s) d\Sigma, 
\]
i.e., the function $\displaystyle{V_1(r)=\dfrac{e^{\alpha r}}{h(r)^k}\int_{\Sigma\cap B_{r}} h(s) d\Sigma}$ is monotone non-decreasing. This implies
\[
\int_{\Sigma\cap B_r} h(s) d\Sigma \geq e^{-\alpha(r-r_0)}\left(\dfrac{h(r)}{h(r_0)}\right)^k\int_{\Sigma\cap B_{r_0}}h(s) d\Sigma,
\]
for every $r>r_0.$ Since $h$ is an increasing function, we have
\[
h(r)|\Sigma\cap B_r|\geq \int_{\Sigma\cap B_r} h(s) d\Sigma \geq e^{-\alpha(r-r_0)}\left(\dfrac{h(r)}{h(r_0)}\right)^k\int_{\Sigma\cap B_{r_0}}h(s) d\Sigma.
\]
This proves the estimate (\ref{est.mono-1}) and concludes the proof of the item (i) of Theorem \ref{mono}.

\emph{Conclusion of the proof of the item (ii).} The identity (\ref{integral}) gives
\[
\int_\Sigma \lambda(R-r)h'(r)d\Sigma = \frac{1}{k}\int_\Sigma \lambda'(R-r)h(r)|\n_\Sigma r|^2 d\Sigma - \int_\Sigma \lambda(R-r)h(r)\lan\vec{H},\n r\ran d\Sigma.
\]
This implies
\begin{equation}\label{eq.mmmm1}
\begin{split}
\dfrac{d}{dR}\left(\dfrac{1}{h(R)^k}\int_\Sigma \lambda(R-r)h'(r)d\Sigma\right) &=-\dfrac{kh'(R)}{h(R)^{k+1}}\int_\Sigma \lambda(R-r)h'(r)d\Sigma\\
&\qquad + \frac{1}{h(R)^k}\int_\Sigma \lambda'(R-r)h'(r)d\Sigma\\
&=-\dfrac{h'(R)}{h(R)^{k+1}}\int_\Sigma \lambda'(R-r)h(r)|\n_\Sigma r|^2 d\Sigma\\
&\qquad +\dfrac{kh'(R)}{h(R)^{k+1}}\int_\Sigma \lambda(R-r)h(r)\lan\vec{H},\n r\ran d\Sigma \\
&\qquad + \frac{1}{h(R)^k}\int_\Sigma \lambda'(R-r)h'(r)d\Sigma.\\
\end{split}
\end{equation}
On the other hand, by using the hypothesis that $\dfrac{h(r)}{h'(r)}$ is non-decreasing, i.e., $-\dfrac{h(r)}{h'(r)}$ is non-increasing, we have

\begin{equation}\label{eq.mmmm2}
\begin{split}
-\dfrac{h'(R)}{h(R)^{k+1}}\int_\Sigma \lambda'(R-r)h(r)&|\n_\Sigma r|^2d\Sigma +\frac{1}{h(R)^k}\int_\Sigma \lambda'(R-r)h'(r)d\Sigma\\
& = \frac{1}{h(R)^k}\left[-\dfrac{h'(R)}{h(R)}\int_\Sigma \lambda'(R-r)h(r)|\n_\Sigma r|^2 d\Sigma\right.\\
&\left.\qquad + \int_\Sigma \lambda'(R-r)h'(r)d\Sigma\right]\\
&\geq\frac{1}{h(R)^k}\left[-\dfrac{h'(R)}{h(R)}\frac{h(R)}{h'(R)}\int_\Sigma \lambda'(R-r)h'(r)|\n_\Sigma r|^2d\Sigma \right.\\
&\left.\qquad + \int_\Sigma \lambda'(R-r)h'(r)d\Sigma\right]\\
&=\dfrac{1}{h(R)^k}\int_\Sigma \lambda'(R-r)[1-|\n_\Sigma r|^2]h'(r)d\Sigma.\\
\end{split}
\end{equation}
Thus, replacing (\ref{eq.mmmm2}) in the right hand side of (\ref{eq.mmmm1}), we obtain
\[
\begin{split}
\dfrac{d}{dR}\left(\dfrac{1}{h(R)^k}\int_\Sigma \lambda(R-r)h'(r)d\Sigma\right)&\geq \dfrac{1}{h(R)^k}\int_\Sigma \lambda'(R-r)[1-|\n_\Sigma r|^2]h'(r)d\Sigma\\
&\qquad+ \dfrac{kh'(R)}{h(R)^{k+1}}\int_\Sigma \lambda(R-r)h(r)\lan\vec{H},\n r\ran d\Sigma\\
&\geq \dfrac{kh'(R)}{h(R)^{k+1}}\int_\Sigma \lambda(R-r)h(r)\lan\vec{H},\n r\ran d\Sigma.
\end{split}
\]
Considering a sequence of functions $\lambda(t)$ converging to the characteristic function of $[0,\infty),$ we obtain
\begin{equation}\label{mono-partial2}
\dfrac{d}{dR}\left(\dfrac{1}{h(R)^k}\int_{\Sigma\cap B_R} h'(r) d\Sigma\right) \geq \dfrac{kh'(R)}{h(R)^{k+1}}\int_{\Sigma\cap B_R} h(r)\lan\vec{H},\n r\ran d\Sigma.\\
\end{equation}
By using the hypothesis $k|\vec{H}|\leq \alpha$ and that $-\dfrac{h(r)}{h'(r)}$ is non-increasing, we have
\[
\begin{split}
\dfrac{d}{dR}\left(\dfrac{1}{h(R)^k}\int_{\Sigma\cap B_R} h'(r) d\Sigma\right)&\geq -\alpha\dfrac{h'(R)}{h(R)^{k+1}}\int_{\Sigma\cap B_R} h(r) d\Sigma\\
&\geq -\alpha\dfrac{h'(R)}{h(R)^{k+1}}\cdot\frac{h(R)}{h'(R)}\int_{\Sigma\cap B_R} h(r)\cdot\frac{h'(r)}{h(r)} d\Sigma\\
&=-\alpha\frac{1}{h(R)^k}\int_{\Sigma\cap B_R} h'(r) d\Sigma,\\
\end{split}
\]
i.e.,
\begin{equation}\label{mono-final-2}
\dfrac{d}{dR}\log\left(\dfrac{1}{h(R)^k}\int_{\Sigma\cap B_R} h'(r) d\Sigma\right) \geq -\alpha.
\end{equation}
Now, let $r_0,r_1\in I$ such that $r_0<r_1.$ Integrating (\ref{mono-final-2}) from $r_0$ to $r_1,$ we obtain
\[
\dfrac{e^{\alpha r_0}}{h(r_0)^k}\int_{\Sigma\cap B_{r_0}} h'(r) d\Sigma \leq
\dfrac{e^{\alpha r_1}}{h(r_1)^k}\int_{\Sigma\cap B_{r_1}} h'(r) d\Sigma, 
\]
i.e., the function $\displaystyle{V_2(r)=\dfrac{e^{\alpha r}}{h(r)^k}\int_{\Sigma\cap B_r} h'(s)d\Sigma}$ is monotone non-decreasing. This implies
\[
\int_{\Sigma\cap B_r} h'(s) d\Sigma \geq e^{-\alpha(r-r_0)}\left(\dfrac{h(r)}{h(r_0)}\right)^k\int_{\Sigma\cap B_{r_0}}h'(s) d\Sigma,
\]
for every $r>r_0.$ If there exists $B>0$ such that $h'(r)<B,$ then
\[
B|\Sigma\cap B_r|\geq \int_{\Sigma\cap B_r} h'(s) d\Sigma \geq e^{-\alpha(r-r_0)}\left(\dfrac{h(r)}{h(r_0)}\right)^k\int_{\Sigma\cap B_{r_0}}h'(s) d\Sigma.
\]
This proves the inequality (\ref{est.mono-2}). Analogously, if $h''(r)>0,$ then 
\[
h'(r)|\Sigma\cap B_r|\geq \int_{\Sigma\cap B_r} h'(s) d\Sigma \geq e^{-\alpha(r-r_0)}\left(\dfrac{h(r)}{h(r_0)}\right)^k\int_{\Sigma\cap B_{r_0}}h'(s) d\Sigma.
\]
This proves the inequality (\ref{est.mono-3}) and concludes the proof of the item (ii) of Theorem \ref{mono}.
\end{proof}

In the following, we prove the corollaries of Theorem \ref{mono}.

\begin{proof}[Proof of Corollary \ref{mono-cones}.] The proof follows immediately from the inequalities (\ref{mono-partial}) and (\ref{mono-partial2}). 
\end{proof}

\begin{proof}[Proof of Corollary \ref{mono-SS}.] First notice that, since \[h'(r)=\sqrt{1-mh(r)^{2-n} - ch(r)^2},\] we have
\[
\frac{d}{dr}\left(\dfrac{h(r)}{h'(r)}\right) =\dfrac{1-\frac{mn}{2}h(r)^{2-n}}{1-mh(r)^{2-n} - ch(r)^2}>0 
\]
for $h(r)>\left(\dfrac{mn}{2}\right)^{\frac{1}{n-2}}.$ If $c>0,$ the condition $\frac{n^n}{4(n-2)^{n-2}}m^2c^{n-2}<1$ implies
$\left(\dfrac{mn}{2}\right)^{\frac{1}{n-2}}\in(s_0,s_1).$ Since $h'(r) = \sqrt{1-mh(r)^{2-n} - ch(r)^2}\leq 1$ for $c>0,$ by using the inequality (\ref{est.mono-2}), we have
\[
|\Sigma\cap B_r| \geq C_2(r_0)e^{-\alpha r}h(r)^k,
\]
for every $r>r_0.$ This proves the item (i) of Corollary \ref{mono-SS}. The estimate of the item (ii) is an immediate consequence of the inequality (\ref{est.mono-1}). If $c<0$ and $\alpha<k-1,$ then the asymptotic expansion
\[
h(r)=\dfrac{1}{\sqrt{-c}}\sinh(\sqrt{-c}r) + \dfrac{m}{2n\sqrt{-c}}\sinh^{1-n}(\sqrt{-c}r) + O(\sinh^{-n-1}(\sqrt{-c}r))
\]
implies that $|\Sigma\cap B_r|$ has at least exponential volume growth at infinity. This proves the item (ii) of Corollary \ref{mono-SS}.
\end{proof}

In order to prove Corollary \ref{mono-RN}, we will explore the asymptotic behaviour of the warping function for the Reissner-Nordstrom manifold.

\begin{lemma}\label{asymp-RN-1}
The warping function $h$ of the Reissner-Nordstrom manifold $M^n$, $n\geq 4,$ satisfies
\begin{equation}
h(r)=r + \frac{m}{2(n-3)}r^{3-n} + O(r^{5-2n}).
\end{equation}
\end{lemma}

\begin{proof}
Define the function
\begin{equation}\label{r(s)}
r(s)= s - \int_s^{\infty}\left(\dfrac{1}{\sqrt{1-mt^{2-n}+q^2t^{4-2n}}} -1\right)dt.
\end{equation}
Since for any $a>b>0$ holds
\[
\frac{1}{\sqrt{a-b}} - \frac{1}{\sqrt{a}} = \frac{b}{\sqrt{a(a-b)}(\sqrt{a-b}+\sqrt{a})},
\]
we have
\[
\begin{split}
&\dfrac{1}{\sqrt{1-mt^{2-n}+q^2t^{4-2n}}}\! -\!1 \!=\! \frac{mt^{2-n} - q^2t^{4-2n}}{\sqrt{1-mt^{2-n}+q^2t^{4-2n}}(\sqrt{1-mt^{2-n}+q^2t^{4-2n}} +1)}\\
											&\qquad\qquad\qquad\qquad= \frac{mt^{2-n}}{2}\dfrac{2(1-\frac{q^2}{m}t^{2-n})}{\sqrt{1-mt^{2-n}+q^2t^{4-2n}}(\sqrt{1-mt^{2-n}+q^2t^{4-2n}} +1)}\\
											&\qquad\qquad\qquad\qquad=\frac{m}{2}t^{2-n} + O(t^{4-2n}).
\end{split}
\]
Thus, the improper integral in (\ref{r(s)}) converges for $n\geq 4$ and $r(s)$ is well defined. Moreover
\begin{equation}\label{r(s)-2}
r(s) = s - \frac{m}{2(n-3)}s^{3-n} + O(s^{5-2n}).
\end{equation}
Notice that the inverse function $s(r)$ satisfies the differential equation \[\dfrac{ds}{dr}=\sqrt{1-ms^{2-n} + q^2s^{4-2n}}.\] Therefore $s(r)=h(r).$ By using (\ref{r(s)-2}), we have
\[
s=r(s) + \frac{m}{2(n-3)}s^{3-n} + O(s^{5-2n}).
\]
This implies
\begin{equation}\label{r(s)-3}
h(r) = s(r) = r + \frac{m}{2(n-3)}r^{3-n} + O(r^{5-2n}),
\end{equation}
which proves the Lemma.
\end{proof}

\begin{proof}[Proof of Corollary \ref{mono-RN}.] By using the proof of Theorem \ref{iso-RN} we have that $\dfrac{h(r)}{h'(r)}$ is non-increasing for $h(r)>s_2.$ On the other hand, $h'(r)<1$ for $h(r)>\left(\frac{q^2}{m}\right)^{\frac{1}{n-2}}.$ Since $\left(\frac{q^2}{m}\right)^{\frac{1}{n-2}}<s_0,$ we have $h'(r)<1$ everywhere. The result then follows from the inequality (\ref{est.mono-2}). 

If $\Sigma$ is a minimal submanifold, then the asymptotic expansion of $h(r)$ in Lemma \ref{asymp-RN-1} implies that $|\Sigma\cap B_r|$ has at least polynomial volume growth at infinity. The result then follows.
\end{proof}

\begin{proof}[Proof of Corollary \ref{mono-space}.] Notice that, by using the hypothesis $h(0)=0$ and $h'(0)=1$ in the first order Taylor expansion of $h(r)$ we have $h(r) = r + R(r),$ where $\displaystyle{\lim_{r\ria 0}\frac{R(r)}{r}=0}.$ This implies that $\displaystyle{\lim_{r\ria 0}\frac{h(r)}{r}=1}$ and thus
\[
V_2(0^+) = \lim_{r\ria0} V_2(r) = \lim_{r\ria0} e^{-\alpha r}\omega_k\frac{r^k}{h(r)^k}\frac{1}{\omega_k r^k}\int_{\Sigma\cap B_r} h'(s)d\Sigma = \omega_k h'(0)=\omega_k.
\]
Since $V_2(r)\geq V_2(0^+)$ we have the result.

\end{proof}

We conclude this section with the proof of Theorem \ref{Theo.2}.

\begin{proof}[Proof of Theorem \ref{Theo.2}.] Let $p\in\Sigma$ and consider the distance function \[r(x)=dist_M(x,p)\] of the ambient space $M^n.$ Let $\{e_1,e_2\}$ be a geodesic frame of $\Sigma$  and $f_{ij}$ be the coefficients of the Hessian matrix of the smooth function $f$ in this frame. Since
\[
(\log h(r))_{ii} = \frac{h''(r)h(r)-h'(r)^2}{h(r)^2}r_i^2 + \frac{h'(r)}{h(r)}r_{ii}
\]
and, by using identity (\ref{hessian-2}) of Lemma \ref{lemma1} for minimal surfaces in warped product manifolds, we have
\begin{equation}\label{eq.theo2-1}
\begin{split}
\Delta_\Sigma\log h(r) &=\dfrac{h''(r)h(r)-h'(r)^2}{h(r)^2}|\n_\Sigma r|^2 + \dfrac{h'(r)}{h(r)}\Delta_\Sigma r\\
                    &=\dfrac{h''(r)h(r)-h'(r)^2}{h(r)^2}|\n_\Sigma r|^2 + \left(\dfrac{h'(r)}{h(r)}\right)^2\left(2-|\n_\Sigma r|^2\right)\\
                    &=2\left(\frac{h'(r)}{h(r)}\right)^2 - |\n_\Sigma r|^2\left(2\left(\frac{h'(r)}{h(r)}\right)^2 - \dfrac{h''(r)}{h(r)}\right).\\
\end{split}
\end{equation}

Now, let us prove the item (i) of Theorem \ref{Theo.2}. If $u(r)=r+\dfrac{h(r)}{h'(r)}$ is non-decreasing, then
\begin{equation}\label{eq.theo2-2}
\begin{split}
\dfrac{d}{dr}\left(r+\dfrac{h(r)}{h'(r)}\right) = \dfrac{2h'(r)^2 - h(r)h''(r)}{h'(r)^2}\geq 0 &\Leftrightarrow 2h'(r)^2 \geq h(r)h''(r)\\
& \Leftrightarrow 2\left(\dfrac{h'(r)}{h(r)}\right)^2\geq \dfrac{h''(r)}{h(r)}\\
& \Leftrightarrow 2\left(\dfrac{h'(r)}{h(r)}\right)^2- \dfrac{h''(r)}{h(r)}\geq 0.
\end{split}
\end{equation}

By using (\ref{eq.theo2-2}) and that $-|\n_\Sigma r|^2\geq -1$ in (\ref{eq.theo2-1}), we have
\begin{equation}\label{3}
\begin{split}
\Delta_\Sigma\log h(r) &=2\left(\frac{h'(r)}{h(r)}\right)^2 - |\n_\Sigma r|^2\left(2\left(\frac{h'(r)}{h(r)}\right)^2 - \dfrac{h''(r)}{h(r)}\right)\\
&\geq \dfrac{h''(r)}{h(r)}=-\dfrac{1}{n-1}\ric_M(\n r).
\end{split}
\end{equation}

Notice that $h(0)=0$ implies $\log h(r)$ is not defined in $p\in\Sigma.$  Consider $\Sigma_t=\Sigma - B(p,t),$ where $B(p,t)$ is the extrinsic ball of center $p$ and radius $t.$ Thus, integrating the inequality (\ref{3}) above over $\Sigma_t,$ we have
\begin{equation}\label{eq.theo2-3}
\begin{split}
-\frac{1}{n-1} \int_{\Sigma_t} \ric_M(\n r)d\Sigma_t &\leq \int_{\Sigma_t} \Delta_\Sigma \log h(r) d\Sigma_t\\
& = - \int_{\partial \Sigma_t} \dfrac{h'(r)}{h(r)}\dfrac{\partial r}{\partial\nu} dS_{\Sigma_t} + \int_{\partial\Sigma} \dfrac{h'(r)}{h(r)}\dfrac{\partial r}{\partial\nu}dS_{\Sigma},
\end{split}
\end{equation}
where we are using the abuse of notation $\partial\Sigma_t = \Sigma\cap\partial B(p,t)$ and $\nu$ is the outward unit normal vector field of $\Sigma_t.$ The Taylor expansion of $h(t)$ near $0,$
\[
h(t) = h(0)+h'(0)t + R(t), \ \lim_{t\ria 0} \frac{R(t)}{t}=0,
\]
and the hypothesis $h(0)=0$ and $h'(0)=1$ gives
\begin{equation}\label{limit}
\lim_{t\ria 0} \dfrac{h(t)}{t}=1.
\end{equation}
Taking $t\ria0$ in (\ref{eq.theo2-3}), by using (\ref{limit}), $h'(0)=1,$ and $\displaystyle{\lim_{t\ria 0}\frac{\partial r}{\partial\nu}=1},$ we have
\[
\begin{split}
-\dfrac{1}{n-1}\int_\Sigma \ric_M(\n r)d\Sigma &\leq -\lim_{t\ria0} \int_{\partial \Sigma_t} \dfrac{h'(r)}{h(r)}\dfrac{\partial r}{\partial\nu}dS_{\Sigma_t} + \int_{\partial\Sigma} \dfrac{h'(r)}{h(r)}\dfrac{\partial r}{\partial\nu}dS_{\Sigma}\\
&=-\lim_{t\ria 0} h'(t)\frac{t}{h(t)}\dfrac{1}{t}\int_{\partial\Sigma_t} \dfrac{\partial r}{\partial\nu}dS_{\Sigma_t} + \int_{\partial\Sigma} \dfrac{h'(r)}{h(r)}\dfrac{\partial r}{\partial\nu}dS_{\Sigma}\\
&= -2\pi + \int_{\partial\Sigma} \dfrac{h'(r)}{h(r)}\dfrac{\partial r}{\partial\nu}dS_{\Sigma}.\\
\end{split}
\]
Since $h$ depends on the choice of $p\in\Sigma,$ $\displaystyle{\int_{\partial\Sigma} \dfrac{h'(r)}{h(r)}\dfrac{\partial r}{\partial\nu}}$ is a function of $p.$ Ranging $p$ over $\Sigma,$ integrating over $\Sigma,$ and by using Fubini's theorem, we have

\[
\begin{split}
-\dfrac{A}{n-1}\int_\Sigma \ric_M(\n r)d\Sigma &\leq -2\pi A + \int_\Sigma\int_{\partial\Sigma} \dfrac{h'(r)}{h(r)}\dfrac{\partial r}{\partial\nu}dS_{\Sigma}d\Sigma\\
                                & \leq -2\pi A + \int_{\partial\Sigma}\int_\Sigma\dfrac{h'(r)}{h(r)}d\Sigma dS_{\Sigma}\\
                                & \leq -2\pi A + \int_{\partial\Sigma}\int_\Sigma \Delta_\Sigma r d\Sigma dS_{\Sigma}\\
                                &=-2\pi A + \int_{\partial\Sigma}\int_{\partial\Sigma}\dfrac{\partial r}{\partial\nu} dS_{\Sigma}dS_{\Sigma}\\
                                &\leq -2\pi A + \int_{\partial\Sigma}\int_{\partial\Sigma} 1dS_{\Sigma}dS_{\Sigma} \\
                                &= -2\pi A + L^2,
\end{split}          
\]
where we used that $\dfrac{\partial r}{\partial\nu}\leq 1$ and that $\Delta_\Sigma r = \dfrac{h'(r)}{h(r)}(2-|\n_\Sigma r|^2)\geq \dfrac{h'(r)}{h(r)}.$ This proves the item (i) of Theorem \ref{Theo.2}.

In order to prove the item (ii), notice that, analogously to (\ref{eq.theo2-2})
\[
\dfrac{d}{dr}\left(r+\dfrac{h(r)}{h'(r)}\right)\leq 0 \Leftrightarrow 2\left(\dfrac{h'(r)}{h(r)}\right)^2- \dfrac{h''(r)}{h(r)}\leq 0.
\]
Thus the inequality (\ref{eq.theo2-1}) becomes
\begin{equation}\label{4}
\begin{split}
\Delta_\Sigma\log h(r) &=2\left(\frac{h'(r)}{h(r)}\right)^2 - |\n_\Sigma r|^2\left(2\left(\frac{h'(r)}{h(r)}\right)^2 - \dfrac{h''(r)}{h(r)}\right)\\
&\geq 2\left(\dfrac{h'(r)}{h(r)}\right)^2.
\end{split}
\end{equation}
On the other hand, using that $\scal_N\geq0,$ we have
\[
\begin{split}
\scal_{M} &=  \dfrac{\scal_N-(n-1)(n-2)h'(r)^2}{h^2} - 2(n-1)\frac{h''(r)}{h(r)}\\
&\geq -(n-1)(n-2)\left(\dfrac{h'(r)}{h(r)}\right)^2 + 2\ric_M(\n r),
\end{split}
\]
i.e.,
\[
\left(\frac{h'(r)}{h(r)}\right)^2 \geq -\frac{1}{(n-1)(n-2)}\left(\scal_M - 2\ric_M(\n r)\right).
\]
Replacing the inequality above in (\ref{4}), we obtain
\[
\Delta_\Sigma\log h(r) \geq -\frac{2}{(n-1)(n-2)}\left(\scal_M - 2\ric_M(\n r)\right).
\]
The rest of the proof of the item (ii) is analogous to the proof of the item (i). This concludes the proof of Theorem \ref{Theo.2}.
\end{proof}

\section{Appendix}

The metric $g=dr^2+h(r)^2g_N$ of a warped product $M=(0,b)\times \s^{n-1}$ such that $h(0)=h(b)=0$ is smooth at $0$ if and only if $h(0)=0,$ $h'(0)=1$ and all the even order derivatives are zero, i.e., $h^{(2m)}(0)=0, m>0.$ The metric is smooth at $b$ if and only if $h(b)=0,$ $h'(b)=-1,$ and all the even order derivatives are zero, i.e., $h^{(2m)}(b)=0, m>0,$ see \cite{Petersen}, Proposition 1 and Proposition 2, p. 13. Otherwise, the metric is singular at the respective extremal point. In particular, if $h(r)$ is an odd function of $r,$ then by the Taylor expansion of $h(r)$ near zero, we have that $h^{(2m)}(0)=h^{(2m)}(b)=0, m>0.$

Below we give three examples of smooth warped product manifolds which satisfy the conditions of Theorem \ref{Theo.2} and Theorem \ref{mono}.

\begin{example}\label{ex-theo.2-1}
{\normalfont Let $B\neq 0$ and $p>0$ be real numbers. Let $I=(0,\infty)$ for $B>0$ and $I=(0,(-B)^{-1/p})$ for $B<0.$ Define $h:I\ria\R$ by
\[
h(r)=r+Br^{p+1}.
\]
We introduce the warped product manifold $M^n(B)=I\times \s^{n-1},$ with the metric $g=dr^2+h(r)^2g_{\s^{n-1}}.$ Clearly, $h(0)=0$ and $h'(0)=1.$ 

Since $h'(r)=1+B(p+1)r^p,$ we have $h'(r)>0$ for $r\in(0,\infty),$ when $B>0,$ and for 
\[
r\in(0,((p+1)(-B))^{-1/p})\subset(0,(-B)^{-1/p}),
\]
when $B<0.$ Thus $h'(r)>0$ everywhere in $M^n(B)$ when $B>0,$ and for \[M^n_+(B)=(0,((p+1)(-B))^{-1/p})\times\s^{n-1}\] when $B<0.$ On the other hand,
\[
u(r)=r+\dfrac{h(r)}{h'(r)} = \dfrac{r(2+(p+2)Br^p)}{1+B(p+1)r^p}
\]
and
\[ 
u'(r)= \dfrac{B^2(p+1)(p+2)r^{2p}-B(p+1)(p-4)r^p+2}{(1+B(p+1)r^p)^2}.
\]
In order to analyse the sign of $u'(r)=\dfrac{d}{dr}\left(r+\dfrac{h(r)}{h'(r)}\right),$ notice that the expression $B^2(p+1)(p+2)r^{2p}-B(p+1)(p-4)r^p+2$ is a quadratic function of $t=r^p.$ The function 
\[
f(t)=B^2(p+1)(p+2)t^2-B(p+1)(p-4)t+2
\]
has the roots
\[
r_0=\dfrac{p-4-p\sqrt{\dfrac{p-7}{p+1}}}{2B(p+2)} \ \mbox{and} \ r_1=\dfrac{p-4+p\sqrt{\dfrac{p-7}{p+1}}}{2B(p+2)}.
\]
Thus we have $u'(r)>0$ everywhere in $M^n_+(B),$ for $B<0$ and, if $p<7,$ $u'(r)>0$ everywhere in $M^n(B), B>0.$ If $p>7$ and $B>0,$ then $u'(r)>0$ for $r\in(0,r_0)\cup(r_1,\infty).$

These metrics are smooth at $0$ for even $p,$ since in this case $h(r)$ is an odd function.
}
\end{example}

\begin{example}[Asymptotically cylindrical manifolds]\label{cylindrical}
{\normalfont
Let $M^n=(0,\infty)\times\s^{n-1}$ with the metric $g=dr^2+h(r)^2g_{\s^{n-1}},$ where $h:[0,\infty)\ria\R$ is given by
\[
h(r)=\dfrac{1}{K}\arctan(Kr).
\]
We have $h(0)=0, \ h'(0)=1,$
\[
u(r)=2+\frac{1}{K}(1+K^2r^2)\arctan(Kr), \ \mbox{and} \ u'(r)= 2+ 2Kr\arctan(Kr)>2>0.
\]
Notice that, since
\[
\dfrac{d}{dr}\left(\dfrac{h(r)}{h'(r)}\right)=1+2Kr\arctan(Kr),
\]
this manifold also satisfies the hypothesis of Theorem \ref{mono}. Since $h(r)$ is an odd function, the metric is smooth at $0.$
More generally, let $h:[0,\infty)\ria \R$ such that $h(0)=0, \ h'(0)=1,$ $h'(r)>0,$ $h''(r)<0$ for all $r\in[0,\infty),$
\[
\lim_{r\ria\infty}h(r)=K>0, \ \mbox{and}\ \lim_{r\ria\infty}h'(r)=\lim_{r\ria\infty}h''(r)=0.
\]
We have
\[
u'(r)=2-\frac{h(r)h''(r)}{h'(r)^2}>2>0.
\]
Since 
\[
\dfrac{d}{dr}\left(\dfrac{h(r)}{h'(r)}\right)= 1-\frac{h(r)h''(r)}{h'(r)^2}>1>0,
\]
these manifolds satisfies also the hypothesis of Theorem \ref{mono}.
We call these manifolds ``asymptotically cylindrical'' because the sectional curvatures satisfy
\[
\lim_{r\ria\infty} K_M(X,Y)=\lim_{r\ria\infty}\dfrac{1-h'(r)^2}{h(r)^2}=\dfrac{1}{K^2}
\]
and
\[
\lim_{r\ria\infty} K_M(X,\n r) = -\lim_{r\ria\infty}\dfrac{h''(r)}{h(r)}=0,
\]
for every $X,Y\in TM, \ X\perp\n r$ and $Y\perp \n r.$

Another example of asymptotically cylindrical manifold is given by $M^n=(0,\infty)\times\s^{n-1}$ with the metric $g=dr^2+h(r)^2g_{\s^{n-1}},$ where
\[
h(r)=\dfrac{r}{(1+ar^p)^{1/p}}, \ a>0, \ p>0.
\]
We have $h(0)=0, \ h'(0)=1,$  
\[h'(r)=\dfrac{1}{(1+ar^p)^{1+1/p}}>0,\ \mbox{and}\ h''(r)=-\dfrac{(p+1)ar^{p-1}}{(1+ar^p)^{2+1/p}}<0.
\] 
This implies that
\[
\lim_{r\ria\infty} h(r)= \dfrac{1}{a^{1/p}}, \ \mbox{and}\ \lim_{r\ria\infty} h'(r)=\lim_{r\ria\infty} h''(r)=0,
\]
and thus $M^n$ is asymptotically cylindrical. Since, for even $p,$ $h(r)$ is an odd function, the metric is also smooth at $0$ for even $p.$
}
\end{example}

\begin{example}\label{ex-log}
{\normalfont
Let $M^n=(0,\infty)\times\s^{n-1}$ with the metric $g=dr^2+h(r)^2g_{\s^{n-1}},$ where $h:(0,\infty)\rightarrow\R$ is given by
\[
h(r)= r\ln(ar^2+e), \ a>0,
\]
and $e$ is the  basis of the natural logarithm. We have that $h(0)=0, \ h'(0)=1,$ $h'(r)=\ln(ar^2+e)+\dfrac{2ar^2}{ar^2+e}>0,$ and
\[
\begin{split}
\dfrac{d}{dr}\left(\dfrac{h(r)}{h'(r)}\right) &= \dfrac{a^2r^4(\ln^2(ar^2+e)+2\ln(ar^2+e)+4) +e^2\ln^2(ar^2+e)}{((ar^2+e)\ln(ar^2+e)+2ar^2)^2}\\
&\qquad+\dfrac{2aer^2\ln(ar^2+e)(\ln(ar^2+e)-1)}{((ar^2+e)\ln(ar^2+e)+2ar^2)^2}>0.
\end{split}
\]
Thus $M^n$ satisfies the hypothesis of Theorem \ref{mono} and the item (i) of Theorem \ref{Theo.2}. Moreover, since $h(r)$ is a odd function, the metric $g$ is smooth at $0.$
}
\end{example}

\begin{remark}
{\normalfont
We can construct many more examples by considering $h(r)=rf(r),$ where $f(r)$ is a positive function which satisfies $f(0)=1.$ In this case $h(0)=0$ and $h'(0)=1.$ If we choose an even function $f(r),$ the metric is also smooth at $0.$ Since $h'(r)=f(r)+rf'(r),$ if we consider $f'(r)\geq0,$ then we have trivially $h'(r)>1>0.$ Notice also that, conversely, by using Taylor expansion of $h(r)$ near $0$, the conditions $h(0)=0$ and $h'(0)=1$ imply the existence of a function $f(r)$ such that $h(r)=rf(r)$ in the interval of convergence of the Taylor expansion.
}
\end{remark}


\begin{bibdiv}
\begin{biblist}



\bib{Bessa}{article}{
   author={Bessa, G. P.},
   author={Garc{\'{\i}}a-Mart{\'{\i}}nez, S. C.},
   author={Mari, L.},
   author={Ramirez-Ospina, H. F.},
   title={Eigenvalue estimates for submanifolds of warped product spaces},
   journal={Math. Proc. Cambridge Philos. Soc.},
   volume={156},
   date={2014},
   number={1},
   pages={25--42},
   issn={0305-0041},
   review={\MR{3144209}},
   doi={10.1017/S0305004113000443},
}


\bib{B-ON}{article}{
   author={Bishop, R. L.},
   author={O'Neill, B.},
   title={Manifolds of negative curvature},
   journal={Trans. Amer. Math. Soc.},
   volume={145},
   date={1969},
   pages={1--49},
   issn={0002-9947},
   review={\MR{0251664 (40 \#4891)}},
}

\bib{Morgan}{article}{
   author={Bray, Hubert},
   author={Morgan, Frank},
   title={An isoperimetric comparison theorem for Schwarzschild space and
   other manifolds},
   journal={Proc. Amer. Math. Soc.},
   volume={130},
   date={2002},
   number={5},
   pages={1467--1472},
   issn={0002-9939},
   review={\MR{1879971 (2002i:53073)}},
   doi={10.1090/S0002-9939-01-06186-X},
}

\bib{Brendle}{article}{
   author={Brendle, Simon},
   title={Constant mean curvature surfaces in warped product manifolds},
   journal={Publ. Math. Inst. Hautes \'Etudes Sci.},
   volume={117},
   date={2013},
   pages={247--269},
   issn={0073-8301},
   review={\MR{3090261}},
   doi={10.1007/s10240-012-0047-5},
}

\bib{Brendle-2}{article}{
   author={Brendle, Simon},
   author={Hung, Pei-Ken},
   author={Wang, Mu-Tao},
   title={A Minkowski Inequality for Hypersurfaces in the Anti-de Sitter-Schwarzschild Manifold},
   journal={Comm. Pure Appl. Math.},
   volume={69},
   number={1},
   date={2016},
   pages={124--144},
   issn={1097-0312},
   doi={10.1002/cpa.21556},
}


\bib{Brendle-Eichmair}{article}{
   author={Brendle, Simon},
   author={Eichmair, Michael},
   title={Isoperimetric and Weingarten surfaces in the Schwarzschild
   manifold},
   journal={J. Differential Geom.},
   volume={94},
   date={2013},
   number={3},
   pages={387--407},
   issn={0022-040X},
   review={\MR{3080487}},
}



\bib{CG-Manusc}{article}{
   author={Choe, Jaigyoung},
   author={Gulliver, Robert},
   title={Isoperimetric inequalities on minimal submanifolds of space forms},
   journal={Manuscripta Math.},
   volume={77},
   date={1992},
   number={2-3},
   pages={169--189},
   issn={0025-2611},
   review={\MR{1188579 (93k:53059)}},
   doi={10.1007/BF02567052},
}




\bib{Gimeno}{article}{
   author={Gimeno, Vicent},
   title={Isoperimetric inequalities for submanifolds. Jellett-Minkowski's
   formula revisited},
   journal={Proc. Lond. Math. Soc. (3)},
   volume={110},
   date={2015},
   number={3},
   pages={593--614},
   issn={0024-6115},
   review={\MR{3342099}},
   doi={10.1112/plms/pdu053},
}


\bib{heintze}{article}{
   author={Heintze, Ernst},
   title={Extrinsic upper bounds for $\lambda_1$},
   journal={Math. Ann.},
   volume={280},
   date={1988},
   number={3},
   pages={389--402},
   issn={0025-5831},
   review={\MR{936318}},
   doi={10.1007/BF01456332},
}

\bib{hsiung}{article}{
   author={Hsiung, Chuan-Chih},
   title={Some integral formulas for closed hypersurfaces},
   journal={Math. Scand.},
   volume={2},
   date={1954},
   pages={286--294},
   issn={0025-5521},
   review={\MR{0068236}},
}

\bib{LT}{article}{
   author={Lu, Jin},
   author={Tanaka, Minoru},
   title={On the compact minimal submanifold in Riemannian manifolds},
   journal={Proc. Sch. Sci. Tokai Univ.},
   volume={35},
   date={2000},
   pages={33--40},
   issn={0919-5025},
   review={\MR{1761525}},
}


\bib{Montiel}{article}{
   author={Montiel, Sebasti{\'a}n},
   title={Unicity of constant mean curvature hypersurfaces in some
   Riemannian manifolds},
   journal={Indiana Univ. Math. J.},
   volume={48},
   date={1999},
   number={2},
   pages={711--748},
   issn={0022-2518},
   review={\MR{1722814 (2001f:53131)}},
   doi={10.1512/iumj.1999.48.1562},
}

\bib{myers}{article}{
   author={Myers, S. B.},
   title={Curvature of closed hypersurfaces and non-existence of closed
   minimal hypersurfaces},
   journal={Trans. Amer. Math. Soc.},
   volume={71},
   date={1951},
   pages={211--217},
   issn={0002-9947},
   review={\MR{0044884}},
}

\bib{Petersen}{book}{
   author={Petersen, Peter},
   title={Riemannian geometry},
   series={Graduate Texts in Mathematics},
   volume={171},
   publisher={Springer-Verlag, New York},
   date={1998},
   pages={xvi+432},
   isbn={0-387-98212-4},
   review={\MR{1480173}},
   doi={10.1007/978-1-4757-6434-5},
}

\bib{reilly}{article}{
   author={Reilly, Robert C.},
   title={On the first eigenvalue of the Laplacian for compact submanifolds
   of Euclidean space},
   journal={Comment. Math. Helv.},
   volume={52},
   date={1977},
   number={4},
   pages={525--533},
   issn={0010-2571},
   review={\MR{0482597}},
}



\bib{Seo}{article}{
   author={Seo, Keomkyo},
   title={Isoperimetric inequalities for submanifolds with bounded mean
   curvature},
   journal={Monatsh. Math.},
   volume={166},
   date={2012},
   number={3-4},
   pages={525--542},
   issn={0026-9255},
   review={\MR{2925153}},
   doi={10.1007/s00605-011-0332-2},
}

\bib{shiorama}{article}{
   author={Shiohama, Katsuhiro},
   title={Minimal immersions of compact Riemannian manifolds in complete and
   non-compact Riemannian manifolds},
   journal={K\=odai Math. Sem. Rep.},
   volume={22},
   date={1970},
   pages={77--81},
   issn={0023-2599},
   review={\MR{0266112}},
}




\bib{Xia-Wu-1}{article}{
   author={Wu, Jie},
   author={Xia, Chao},
   title={On rigidity of hypersurfaces with constant curvature functions in
   warped product manifolds},
   journal={Ann. Global Anal. Geom.},
   volume={46},
   date={2014},
   number={1},
   pages={1--22},
   issn={0232-704X},
   review={\MR{3205799}},
   doi={10.1007/s10455-013-9405-x},
}

\bib{Xia-Wu-2}{article}{
   author={Wu, Jie},
   author={Xia, Chao},
   title={Hypersurfaces with constant curvature quotients in warped product
   manifolds},
   journal={Pacific J. Math.},
   volume={274},
   date={2015},
   number={2},
   pages={355--371},
   issn={0030-8730},
   review={\MR{3332908}},
   doi={10.2140/pjm.2015.274.355},
}

\bib{Yau}{article}{
   author={Yau, Shing-Tung},
   title={Isoperimetric constants and the first eigenvalue of a compact manifold},
   journal={Ann. Sci. Ecole Norm. Sup.},
   volume={8},
   date={1975},
   pages={487--507},
}


\end{biblist}
\end{bibdiv}
\end{document}